\documentclass[11pt]{article}
\usepackage[margin=1in]{geometry}
\usepackage{amsmath,amssymb,amsthm,mathtools}
\usepackage{graphicx,booktabs}
\usepackage{algorithm,algpseudocode}
\usepackage[dvipsnames]{xcolor}
\usepackage[colorlinks=true,linkcolor=blue!60!black,citecolor=blue!60!black,urlcolor=blue!60!black]{hyperref}
\usepackage{microtype}
\graphicspath{{./images/}{./}}

\newtheorem{theorem}{Theorem}[section]
\newtheorem{lemma}[theorem]{Lemma}

\newtheorem{corollary}[theorem]{Corollary}
\theoremstyle{definition}

\newtheorem{assumption}[theorem]{Assumption}
\theoremstyle{remark}
\newtheorem{remark}[theorem]{Remark}

\newcommand{\R}{\mathbb{R}}
\newcommand{\Om}{\Omega}
\newcommand{\Dp}{D_p}
\newcommand{\Ds}{D_s}
\newcommand{\Gam}{\Gamma}
\newcommand{\eps}{\epsilon}
\newcommand{\ep}{\epsilon_p}
\newcommand{\es}{\epsilon_s}
\newcommand{\Phit}{\widehat{\Phi}}
\newcommand{\Phih}{\Phi_h}
\newcommand{\PhiH}{\Phi_H}
\newcommand{\Phitg}{\Phi_h^{\mathrm{tg}}}
\newcommand{\utg}{u_h^{\mathrm{tg}}}
\newcommand{\Etau}{E_\tau}
\newcommand{\rt}{\rho_\tau}
\newcommand{\Pt}{P_\tau}
\newcommand{\Jt}{J_\tau}
\newcommand{\norm}[1]{\left\|#1\right\|}
\newcommand{\enorm}[1]{|\!|\!|#1|\!|\!|}
\newcommand{\Hs}{H^{s}(\Dp\cup\Ds)}
\newcommand{\Hsig}{H^{\sigma}(\Dp\cup\Ds)}
\newcommand{\dd}{\,\mathrm{d}}
\newcommand{\Vh}{V_h}
\newcommand{\VH}{V_H}
\newcommand{\Vhz}{V_h^0}
\newcommand{\VHz}{V_H^0}
\newcommand{\Zmax}{Z_{\max}}
\newcommand{\MU}{M_U}

\DeclareMathOperator{\dist}{dist}

\begin{document}

\title{A Two-Grid Finite Element Method for the Nonlinear Poisson--Boltzmann Equation Based on a Solution Decomposition}
\author{Zhen Chao\thanks{Department of Mathematics, Western Washington University, Bellingham, WA 98225, USA. Corresponding author: \texttt{chaoz2@wwu.edu}.}
\and Jingchun Chen\thanks{College of Earth and Mineral Sciences, The Pennsylvania State University, University Park, PA 16802, USA.}
\and Cong He\thanks{Department of Mathematical Sciences, University of Wisconsin--Milwaukee, Milwaukee, WI 53211, USA.}}
\date{}
\begingroup\makeatletter\def\@fnsymbol#1{\@arabic{#1}}\makeatother
\maketitle
\endgroup

\begin{abstract}
The nonlinear Poisson--Boltzmann equation (PBE) is a widely used implicit-solvent model for the electrostatics of biomolecules in ionic solvent. Its numerical solution is complicated by the point charges, the dielectric jump across the molecular surface, and the exponential nonlinearity. We study a two-grid finite element method for the PBE built on the solution decomposition $u=G+\Psi+\Phit$ of Xie [J.~Comput.~Phys.~275 (2014) 294--309], in which $G$ carries the singularities analytically, $\Psi$ solves a linear interface problem, and $\Phit$ solves a semilinear interface problem for an arbitrary number of ionic species. The nonlinear problem for $\Phit$ is solved on a coarse mesh only, followed by a single Newton correction on the fine mesh, so that only symmetric positive definite linear systems are solved on the fine mesh. The analysis uses a $C^2$ truncation of the exponential that grows quadratically beyond a level $\tau$; for $\tau$ above an explicit bound the truncated problem has the same solution as the original one, and no discrete maximum principle or mesh-angle condition is needed. Under standard piecewise regularity assumptions we prove quasi-optimal energy and $L^2$ estimates and an $L^4$ estimate for the finite element approximation of $\Phit$, and for the two-grid solution we prove $\enorm{u-\utg}\le C\big(h^{s-1}+H^{2(s-1)+t/2}\big)$ in three dimensions, where $t=\min\{s,\sigma\}-1$. Whenever the dual problem is at least as regular as the primal one ($\sigma\ge s$), the coarse mesh can be chosen as $H\sim h^{2/5}$ without loss of accuracy. A manufactured-solution test confirms the predicted rates and the quadratic dependence of the two-grid correction on the coarse error.
\end{abstract}



\section{Introduction}

The Poisson--Boltzmann equation (PBE) is a standard continuum model for the electrostatic potential of a biomolecule in an ionic solvent \cite{SharpHonig1990,Fogolari2002,Baker2005,LuZhouHolstMcCammon2008}. It is solved by finite difference and finite element packages such as APBS \cite{Baker2001}, DelPhi \cite{Rocchia2001} and MIBPB \cite{GengYuWei2007}, and by boundary integral methods \cite{LuChengHuangMcCammon2006,GengKrasny2013}. From the numerical point of view the PBE is an elliptic interface problem with three difficulties: the source is a sum of Dirac distributions at the atomic centers, so the potential is singular there; the dielectric coefficient jumps across the molecular surface; and the ionic charge density depends exponentially on the potential. We refer to the surveys \cite{LuZhouHolstMcCammon2008} and \cite{Xie2014} for the literature on each of these issues.

The point-charge singularities are usually removed by writing the potential as a singular part, known in closed form, plus one or two regular parts that solve interface problems with bounded data \cite{Zhou1996,ChernLiuWang2003,GengYuWei2007,ChenHolstXu2007,Holst2012}. In this paper we use the decomposition of Xie \cite{Xie2014},
\[
u = G + \Psi + \Phit,
\]
where $G$ is the Coulomb potential of the solute charges, $\Psi$ solves a linear interface problem that carries the boundary and interface conditions, and $\Phit$ solves a semilinear interface problem with homogeneous interface and boundary conditions. The nonlinearity in the problem for $\Phit$ comes from the ionic charges and acts only in the solvent region. Both $\Psi$ and $\Phit$ solve interface problems with bounded data, and the decomposition allows any number of ionic species. Xie combined it with a finite element method and a modified Newton minimization scheme in a FEniCS-based PBE solver \cite{Xie2014}, and the same splitting has been used in finite element solvers for the Poisson--Nernst--Planck equations \cite{XieChao2020,ChaoXie2021,YingFanLiLu2021}.

In such solvers most of the work goes into the nonlinear algebraic system for the regular part $\Phit$. It is usually solved by a damped Newton or Newton-type minimization method \cite{Holst2012,Xie2014}, and every Newton step requires a large linear solve on the fine mesh. The two-grid method of Xu \cite{Xu1994,Xu1996}, see also \cite{XuZhou2000,XuZhou2001,AxelssonLayton1996,MarionXu1995,DawsonWheelerWoodward1998,JinShuXu2006,BiGinting2011}, reduces this cost. The nonlinear problem is solved only on a coarse mesh with mesh size $H$, and the solution on the fine mesh with mesh size $h\ll H$ is obtained by one linearized (Newton) correction. The coarse mesh can be much coarser than the fine one, typically $H\sim h^{\gamma}$ with $\gamma<1$, without loss of asymptotic accuracy. Holst, Szypowski and Zhu \cite{HolstSzypowskiZhu2013} extended the analysis to semilinear interface problems with discontinuous coefficients, with the regularized PBE as their main example, and showed that $H\sim h^{2/5}$ suffices in three dimensions under full regularity. Two-grid methods have also been used to decouple and accelerate finite element solvers for the Poisson--Nernst--Planck equations \cite{ShenShuYangLu2020,YingFanLiLu2021}.

Here we apply the two-grid method to Xie's decomposition and give an a priori error analysis of the resulting scheme. Two points distinguish this setting from \cite{HolstSzypowskiZhu2013}. First, the nonlinear term in the problem for $\Phit$ is
\[
\rho(U+\Phit)\chi_{\Ds}, \qquad \rho(\xi) = -\beta\sum_{i=1}^{n} Z_i \bar c_i e^{-Z_i\xi},
\]
where $U=G+\Psi$ is the potential of the solute in pure water, $Z_i$ and $\bar c_i$ are the charge numbers and bulk concentrations of the ionic species, and $\chi_{\Ds}$ is the indicator function of the solvent region. This term depends on $x$ through $U$, vanishes in the solute region, and in the discrete problem $U$ is replaced by the finite element approximation $U_h=G+\Psi_h$, so the analysis has to account for this perturbation. Second, for nonlinearities of supercritical growth such as the exponential, the analysis in \cite{HolstSzypowskiZhu2013} relies on a discrete maximum principle, i.e., on a sign condition on the off-diagonal stiffness entries. For tetrahedral meshes generated from molecular surfaces this is a strong restriction. We do not need it. Instead we replace $e^{y}$ by a $C^2$ function $\Etau(y)$ that equals $e^{y}$ for $y\le\tau$ and is quadratic for $y>\tau$; this is a smooth version of the overflow safeguard (``blow-up'' truncation) used in \cite{Xie2014}. The truncated nonlinearity is monotone, grows quadratically and has a bounded second derivative. We prove the bound
\[
\norm{\Phit}_{L^\infty(\Om)}\le \norm{U}_{L^\infty(\Ds)}
\]
for the solution of the original problem and conclude that, for every $\tau$ above an explicit threshold, the truncated problem has the same unique solution. The error analysis can therefore be done for the truncated problem, where no pointwise bounds on discrete solutions are needed, and the results hold for the original problem. As a by-product we obtain a short proof of existence and uniqueness for $\Phit$ that works in $H^1_0(\Om)$ rather than on the set $\mathcal H$ used in \cite{Xie2014}.

The main results are as follows. Under piecewise $H^{s}$ regularity of $\Psi$ and $\Phit$ and $H^\sigma$ regularity of a linear dual problem, $1<s,\sigma\le 2$, we prove that the finite element approximation $\Phih$ of $\Phit$ computed with the perturbed datum $U_h$ satisfies energy-norm, $L^2$ and $L^4$ error estimates, in which the perturbation enters only through $\norm{\Psi-\Psi_h}_{L^2(\Ds)}$ (Theorems~\ref{thm:energy} and \ref{thm:L2}, Corollary~\ref{cor:L4}). For the two-grid solution $\utg = G+\Psi_h+\Phitg$ we prove
\[
\enorm{u-\utg}\le C\big(h^{s-1}+H^{2(s-1)+t/2}\big),\qquad t=\min\{s,\sigma\}-1,
\]
in three dimensions (Theorem~\ref{thm:twogrid}), so that the choice $H\sim h^{(s-1)/(2(s-1)+t/2)}$ keeps the accuracy of the nonlinear fine-mesh solution. If $\sigma\ge s$, in particular if $s=\sigma$, this is $H\sim h^{2/5}$. The estimate holds whether the coarse nonlinear problem uses the fine-mesh datum $U_h$ or the coarse-mesh datum $U_H=G+\Psi_H$; with the second choice the whole nonlinear computation stays on the coarse mesh. On the fine mesh the scheme needs only two linear solves with symmetric positive definite matrices, one for $\Psi_h$ and one for the Newton correction, and both can be done by conjugate gradients with multilevel preconditioners that are robust with respect to the dielectric jump \cite{XuZhu2008}. For a symmetric $1{:}1$ electrolyte the truncated nonlinearity is $\kappa^2\sinh_\tau$, a $C^2$ version of the modified hyperbolic sine of \cite{Xie2014}.

The remainder of the paper is organized as follows. Section~\ref{sec:model} recalls the dimensionless PBE model, the solution decomposition of \cite{Xie2014}, and the weak formulations. Section~\ref{sec:truncation} introduces the truncated nonlinearity and proves the a priori bounds and the equivalence of the truncated and the original problems. Section~\ref{sec:fem} contains the finite element discretization and the error estimates for $\Phih$, including the effect of the perturbed datum. Section~\ref{sec:twogrid} presents the two-grid algorithm, its error analysis and the choice of the coarse mesh size. Section~\ref{sec:numerics} reports a three-dimensional test with a manufactured solution, and Section~\ref{sec:conclusion} contains concluding remarks.

\section{The Poisson--Boltzmann model and its solution decomposition}
\label{sec:model}

\subsection{The dimensionless PBE model}

Let $\Om\subset\R^3$ be a bounded Lipschitz domain that is partitioned as $\overline{\Om} = \overline{\Dp}\cup\overline{\Ds}$, where $\Dp$ is the solute region hosting a biomolecule with $n_p$ atoms, $\Ds=\Om\setminus\overline{\Dp}$ is the solvent region, and $\Gam=\partial\Dp$ is the interface between them. We assume that $\overline{\Dp}\subset\Om$, so that $\Gam\cap\partial\Om=\emptyset$, and that both $\Dp$ and $\Ds$ are Lipschitz domains. The $j$th atom has position $\mathbf r_j\in\Dp$ and charge number $z_j$, and the solvent contains $n$ ionic species with charge numbers $Z_i\neq 0$ and bulk concentrations $\bar c_i\ge 0$ (in mol/L), $i=1,\dots,n$. The two regions are treated as homogeneous dielectric media with relative permittivities $\ep$ and $\es$; throughout we only use $0<\ep\le\es$, and we write $\eps$ for the piecewise constant function with $\eps|_{\Dp}=\ep$ and $\eps|_{\Ds}=\es$. Following \cite{Xie2014}, when the length is measured in \r{A}ngstr\"oms and the potential in units of $k_BT/e_c$, the PBE model takes the dimensionless form
\begin{equation}\label{eq:pbe}
\left\{
\begin{aligned}
-\ep\Delta u(\mathbf r) &= \alpha\sum_{j=1}^{n_p} z_j\delta_{\mathbf r_j}, && \mathbf r\in\Dp,\\
-\es\Delta u(\mathbf r) + \rho(u(\mathbf r)) &= 0, && \mathbf r\in\Ds,\\
u(\mathbf s^+) = u(\mathbf s^-),\quad \es\frac{\partial u(\mathbf s^+)}{\partial\mathbf n(\mathbf s)} &= \ep\frac{\partial u(\mathbf s^-)}{\partial\mathbf n(\mathbf s)}, && \mathbf s\in\Gam,\\
u(\mathbf s) &= \hat g(\mathbf s), && \mathbf s\in\partial\Om,
\end{aligned}
\right.
\end{equation}
where $\delta_{\mathbf r_j}$ is the Dirac distribution at $\mathbf r_j$, $\mathbf n$ is the unit outward normal of $\Dp$, $u(\mathbf s^\pm)=\lim_{t\to 0^+}u(\mathbf s\pm t\mathbf n(\mathbf s))$, $\hat g$ is a given boundary function (commonly $\hat g=0$), and the ionic charge density is described by
\begin{equation}\label{eq:rho}
\rho(\xi) := -\beta\sum_{i=1}^{n} Z_i\bar c_i e^{-Z_i\xi},\qquad \xi\in\R.
\end{equation}
The positive constants $\alpha$ and $\beta$ collect the physical constants; for $T=298.15$~K one has $\alpha\approx 7042.94$ and $\beta\approx 4.2413$ \cite{Xie2014}. We assume throughout that the bulk solution is electrically neutral,
\begin{equation}\label{eq:neutral}
\sum_{i=1}^{n} Z_i\bar c_i = 0,\qquad \bar c:=\sum_{i=1}^{n}\bar c_i>0,
\end{equation}
which is the standard physical assumption (the second condition merely excludes the trivial case of pure water) and implies $\rho(0)=0$. Since
\begin{equation}\label{eq:rhoprime}
\rho'(\xi) = \beta\sum_{i=1}^{n} Z_i^2\bar c_i e^{-Z_i\xi} > 0,
\end{equation}
the function $\rho$ is strictly increasing. For a symmetric $1{:}1$ electrolyte ($n=2$, $Z_1=-Z_2=1$, $\bar c_1=\bar c_2=I_s$) we have $\rho(\xi)=\kappa^2\sinh\xi$ with $\kappa^2=2\beta I_s$, and the equation in $\Ds$ becomes $-\es\Delta u+\kappa^2\sinh u=0$.

\subsection{Solution decomposition}

To remove the singularities caused by the Dirac distributions, Xie \cite{Xie2014} decomposes the solution of \eqref{eq:pbe} as
\begin{equation}\label{eq:decomp}
u = G + \Psi + \Phit \quad\text{in }\Om,
\end{equation}
where $G$ is the Coulomb potential of the solute charges,
\begin{equation}\label{eq:G}
G(\mathbf r) = \frac{\alpha}{4\pi\ep}\sum_{j=1}^{n_p}\frac{z_j}{|\mathbf r-\mathbf r_j|},
\end{equation}
$\Psi$ is the solution of the linear interface problem
\begin{equation}\label{eq:Psi}
\left\{
\begin{aligned}
\Delta\Psi &= 0 && \text{in }\Dp\cup\Ds,\\
\Psi(\mathbf s^+) = \Psi(\mathbf s^-),\quad \es\frac{\partial\Psi(\mathbf s^+)}{\partial\mathbf n} &= \ep\frac{\partial\Psi(\mathbf s^-)}{\partial\mathbf n} + (\ep-\es)\frac{\partial G(\mathbf s)}{\partial\mathbf n} && \text{on }\Gam,\\
\Psi &= \hat g - G && \text{on }\partial\Om,
\end{aligned}
\right.
\end{equation}
and $\Phit$ is the solution of the semilinear interface problem
\begin{equation}\label{eq:Phi}
\left\{
\begin{aligned}
\Delta\Phit &= 0 && \text{in }\Dp,\\
-\es\Delta\Phit + \rho(U+\Phit) &= 0 && \text{in }\Ds,\\
\Phit(\mathbf s^+) = \Phit(\mathbf s^-),\quad \es\frac{\partial\Phit(\mathbf s^+)}{\partial\mathbf n} &= \ep\frac{\partial\Phit(\mathbf s^-)}{\partial\mathbf n} && \text{on }\Gam,\\
\Phit &= 0 && \text{on }\partial\Om,
\end{aligned}
\right.
\end{equation}
with
\begin{equation}\label{eq:U}
U := G+\Psi.
\end{equation}
Since $-\Delta\big(\tfrac{1}{4\pi|\mathbf r-\mathbf r_j|}\big)=\delta_{\mathbf r_j}$, one checks directly that $u$ defined by \eqref{eq:decomp} satisfies \eqref{eq:pbe}; see \cite{Xie2014}. The functions $G$, $\Psi$ and $\Phit$ are the contributions of the solute charges, of the boundary and interface conditions, and of the ionic charges, respectively, and $U=G+\Psi$ is the potential of the solute in pure water. In \cite{Xie2014} the nonlinearity in \eqref{eq:Phi} is written as $-\beta\sum_i Z_i\bar c_i w_i(\mathbf r)e^{-Z_i\Phit}$ with $w_i=e^{-Z_iU}$; this is the same as $\rho(U+\Phit)$. Since all atomic centers $\mathbf r_j$ lie in $\Dp$ and $\dist(\mathbf r_j,\overline{\Ds})>0$, the function $G$ is real-analytic in a neighborhood of $\overline{\Ds}$ and, in particular, $G\in C^\infty(\overline{\Ds})$ and $\nabla G\in L^\infty(\Ds)$.

\subsection{Notation and weak formulations}

We use standard notation for Lebesgue and Sobolev spaces. For a measurable set $D\subseteq\Om$, $\norm{\cdot}_{0,p,D}$ denotes the norm of $L^p(D)$, $1\le p\le\infty$, and $(\cdot,\cdot)_D$ the $L^2(D)$ inner product; the subscript $D$ is omitted when $D=\Om$. For $s>0$ we write $\Hs$ for the space of functions $v$ with $v|_{\Dp}\in H^s(\Dp)$ and $v|_{\Ds}\in H^s(\Ds)$, normed by $\norm{v}^2_{\Hs}=\norm{v}^2_{H^s(\Dp)}+\norm{v}^2_{H^s(\Ds)}$. Let $H^1_0(\Om)=\{v\in H^1(\Om): v=0\text{ on }\partial\Om\}$. We shall use the Poincar\'e and Sobolev inequalities
\begin{equation}\label{eq:poincare}
\norm{v}_{0,2}\le C_P\norm{\nabla v}_{0,2},\qquad \norm{v}_{0,6}\le C\norm{\nabla v}_{0,2},\qquad \forall v\in H^1_0(\Om),
\end{equation}
and the Gagliardo--Nirenberg (Ladyzhenskaya) inequality in three dimensions \cite{Ladyzhenskaya1969}
\begin{equation}\label{eq:GN}
\norm{v}_{0,4}\le C_{L}\norm{v}_{0,2}^{1/4}\norm{\nabla v}_{0,2}^{3/4},\qquad \forall v\in H^1_0(\Om).
\end{equation}
By H\"older's inequality and \eqref{eq:poincare}, $\norm{v}_{0,p}\le C\norm{\nabla v}_{0,2}$ for every $p\in[1,6]$ and $v\in H^1_0(\Om)$.

The bilinear form associated with the dielectric operator is
\begin{equation}\label{eq:a}
a(v,w) := \ep\int_{\Dp}\nabla v\cdot\nabla w\dd\mathbf r + \es\int_{\Ds}\nabla v\cdot\nabla w\dd\mathbf r,\qquad v,w\in H^1(\Om),
\end{equation}
and we define the energy norm $\enorm{v}:=a(v,v)^{1/2}$ on $H^1_0(\Om)$. Clearly
\begin{equation}\label{eq:anorm}
\ep\norm{\nabla v}_{0,2}^2\le\enorm{v}^2\le\es\norm{\nabla v}_{0,2}^2,\qquad a(v,w)\le\enorm{v}\,\enorm{w}.
\end{equation}
Using \eqref{eq:poincare} we have in particular $\norm{v}_{0,p}\le C\enorm{v}$ for $p\in[1,6]$ and $v\in H^1_0(\Om)$.

The weak form of \eqref{eq:Psi} reads \cite{Xie2014}: find $\Psi\in H^1(\Om)$ with $\Psi=\hat g-G$ on $\partial\Om$ such that
\begin{equation}\label{eq:Psiweak}
a(\Psi,v) = (\ep-\es)\int_{\Ds}\nabla G\cdot\nabla v\dd\mathbf r\qquad\forall v\in H^1_0(\Om).
\end{equation}
Here the interface term $(\es-\ep)\int_\Gam \frac{\partial G}{\partial\mathbf n}v\dd s$ has been converted into a volume integral by Green's identity and the harmonicity of $G$ in $\Ds$, which is convenient for complex molecular surfaces. If $\hat g\in H^{1/2}(\partial\Om)$, problem \eqref{eq:Psiweak} has a unique solution by the Lax--Milgram lemma \cite[Theorem~4.1]{Xie2014}.

The weak form of \eqref{eq:Phi} reads: find $\Phit\in\mathcal H$ such that
\begin{equation}\label{eq:Phiweak}
a(\Phit,v) + \big(\rho(U+\Phit),v\big)_{\Ds} = 0\qquad\forall v\in H^1_0(\Om),
\end{equation}
where
\begin{equation}\label{eq:calH}
\mathcal H := \big\{v\in H^1_0(\Om): e^{-Z_iv}\in L^2(\Ds),\ i=1,\dots,n\big\}
\end{equation}
is a subset of $H^1_0(\Om)$ on which the nonlinear term is well defined whenever $U\in L^\infty(\Ds)$. It is shown in \cite[Theorem~4.2]{Xie2014} that \eqref{eq:Phiweak} has a unique solution, which is the unique minimizer of a strictly convex functional on $\mathcal H$. In Section~\ref{sec:truncation} we give a different proof that also produces the a priori bound needed for the error analysis.

Throughout the paper we make the following standing assumption on the data.

\begin{assumption}\label{ass:U}
The function $U=G+\Psi$ defined by \eqref{eq:U} satisfies $U\in L^\infty(\Ds)$. We write
\begin{equation}\label{eq:MU}
\MU := \norm{U}_{0,\infty,\Ds}.
\end{equation}
\end{assumption}

\begin{remark}
Since $G\in C^\infty(\overline{\Ds})$, Assumption~\ref{ass:U} only concerns $\Psi$. If $\Gam$ is of class $C^2$, $\Om$ is convex (for instance a box, the usual choice for the PBE) and $\hat g$ is smooth, then $\Psi\in H^2(\Dp)\cap H^2(\Ds)$ by \cite[Theorem~4.1]{Xie2014} and the regularity theory of elliptic interface problems \cite{NicaiseSandig1994,ChenZou1998}; because $\Ds$ is a Lipschitz domain in $\R^3$, the Sobolev embedding $H^2(\Ds)\hookrightarrow C(\overline{\Ds})$ then gives $U\in C(\overline{\Ds})$. More generally, the assumption holds whenever $\Psi\in H^{s}(\Ds)$ for some $s>3/2$; since the regularity Assumption~\ref{ass:reg} below also allows $1<s\le3/2$, we keep Assumption~\ref{ass:U} as a separate standing assumption. The constant $\MU$ depends only on the linear problem \eqref{eq:Psi}, i.e., on the potential of the solute in pure water, and can be estimated before the nonlinear problem is solved.
\end{remark}

\section{A \texorpdfstring{$C^2$}{C2} truncation of the nonlinearity and a priori bounds}
\label{sec:truncation}

The exponential growth of $\rho$ in \eqref{eq:rho} is the main difficulty in the finite element error analysis of \eqref{eq:Phiweak}: a priori, the Galerkin approximation of $\Phit$ is only bounded in $H^1$, and the nonlinearity is neither globally Lipschitz nor of subcritical growth. In \cite{HolstSzypowskiZhu2013} this is handled by discrete $L^\infty$ bounds, which require a mesh condition guaranteeing a discrete maximum principle. Here we proceed differently, following the idea of the ``blow-up'' test in \cite[Algorithm~2]{Xie2014}: the exponential is replaced by a $C^2$ function of quadratic growth, and we show that this does not change the solution.

\subsection{The truncated exponential}

For $\tau>0$ define
\begin{equation}\label{eq:Etau}
\Etau(y) :=
\begin{cases}
e^{y}, & y\le\tau,\\[2pt]
e^{\tau}\Big(1+(y-\tau)+\tfrac12(y-\tau)^2\Big), & y>\tau.
\end{cases}
\end{equation}
That is, $\Etau$ is the exponential up to the level $\tau$ and its second-order Taylor polynomial beyond it.

\begin{lemma}\label{lem:Etau}
For every $\tau>0$ the function $\Etau$ has the following properties.
\begin{enumerate}
\item[\rm (i)] $\Etau\in C^2(\R)$, $\Etau'(y)>0$ and $0<\Etau''(y)\le e^{\tau}$ for all $y\in\R$; in particular $\Etau$ is strictly increasing and strictly convex.
\item[\rm (ii)] $\Etau(y)=e^{y}$, $\Etau'(y)=e^y$ and $\Etau''(y)=e^y$ for all $y\le\tau$.
\item[\rm (iii)] For all $y\in\R$,
\begin{equation}\label{eq:Egrowth}
0<\Etau(y)\le e^{\tau}(1+|y|)^2,\qquad 0<\Etau'(y)\le e^{\tau}(1+|y|).
\end{equation}
\end{enumerate}
\end{lemma}

\begin{proof}
For $y>\tau$ we have
\[
\Etau'(y)=e^{\tau}\big(1+(y-\tau)\big),\qquad \Etau''(y)=e^{\tau},
\]
while for $y\le\tau$ all derivatives equal $e^{y}$. The one-sided limits of $\Etau$, $\Etau'$ and $\Etau''$ at $y=\tau$ coincide, which gives (i) and (ii). For (iii), if $y\le\tau$ then $\Etau(y)=e^y\le e^\tau$ and $\Etau'(y)\le e^\tau$; if $y>\tau$ then, since $0<y-\tau<y$,
\[
\Etau(y)\le e^\tau\big(1+(y-\tau)\big)^2\le e^\tau(1+y)^2,\qquad \Etau'(y)\le e^\tau(1+y).
\]
\end{proof}

The truncated ionic charge density and its primitive are defined by
\begin{equation}\label{eq:rhotau}
\rt(\xi) := -\beta\sum_{i=1}^{n} Z_i\bar c_i\,\Etau(-Z_i\xi),\qquad
\Pt(\xi) := \int_0^{\xi}\rt(\eta)\dd\eta,\qquad \xi\in\R.
\end{equation}
Let $\Zmax:=\max_i|Z_i|$.

\begin{lemma}\label{lem:rhotau}
Let \eqref{eq:neutral} hold and $\tau>0$. Then $\rt\in C^2(\R)$ and the following hold.
\begin{enumerate}
\item[\rm (i)] $\rt(0)=0$ and $\rt'(\xi)=\beta\sum_{i=1}^{n}Z_i^2\bar c_i\Etau'(-Z_i\xi)>0$ for all $\xi\in\R$; hence $\rt$ is strictly increasing, $\Pt$ is convex and $\Pt\ge 0$.
\item[\rm (ii)] There are constants $C_0,C_1,C_2>0$ depending only on $\tau,\beta,\Zmax$ and $\bar c$ such that for all $\xi\in\R$,
\begin{equation}\label{eq:rhogrowth}
\begin{aligned}
|\rt(\xi)|&\le C_0(1+|\xi|)^2,&\qquad 0<\rt'(\xi)&\le C_1(1+|\xi|),\\
|\rt''(\xi)|&\le C_2,&\qquad 0\le \Pt(\xi)&\le C_0(1+|\xi|)^3.
\end{aligned}
\end{equation}
\item[\rm (iii)] $\rt(\xi)=\rho(\xi)$, $\rt'(\xi)=\rho'(\xi)$ and $\rt''(\xi)=\rho''(\xi)$ whenever $\Zmax|\xi|\le\tau$.
\end{enumerate}
\end{lemma}

\begin{proof}
(i) By \eqref{eq:neutral} and $\Etau(0)=1$, $\rt(0)=-\beta\sum_iZ_i\bar c_i=0$. The formula for $\rt'$ follows from the chain rule and $\rt'>0$ from Lemma~\ref{lem:Etau}(i), since $\bar c_i\ge0$ and $\bar c>0$. Convexity of $\Pt$ follows from $\Pt''=\rt'>0$, and $\Pt\ge0$ from $\Pt(0)=0$, $\Pt'(0)=\rt(0)=0$ and convexity.
(ii) The bounds on $\rt$ and $\rt'$ follow from \eqref{eq:Egrowth} and $(1+|Z_i\xi|)\le(1+\Zmax)(1+|\xi|)$; for instance
\[
C_0=\beta e^{\tau}\Zmax\bar c(1+\Zmax)^2,\qquad C_1=\beta e^\tau\Zmax^2\bar c(1+\Zmax)
\]
will do. Since
\[
\rt''(\xi)=-\beta\sum_{i=1}^{n}Z_i^3\bar c_i\,\Etau''(-Z_i\xi),
\]
Lemma~\ref{lem:Etau}(i) gives $|\rt''(\xi)|\le C_2:=\beta e^\tau\Zmax^3\bar c$. The bound on $\Pt$ follows by integrating the bound on $\rt$.
(iii) If $\Zmax|\xi|\le\tau$ then $-Z_i\xi\le\tau$ for all $i$, and Lemma~\ref{lem:Etau}(ii) applies.
\end{proof}

\begin{remark}
In the symmetric $1{:}1$ case, $\rt(\xi)=\kappa^2\sinh_\tau(\xi)$ and $\rt'(\xi)=\kappa^2\cosh_\tau(\xi)$ with
\begin{equation}\label{eq:sinhtau}
\sinh_\tau(\xi):=\tfrac12\big(\Etau(\xi)-\Etau(-\xi)\big),\qquad
\cosh_\tau(\xi):=\tfrac12\big(\Etau'(\xi)+\Etau'(-\xi)\big).
\end{equation}
They are $C^2$ counterparts of the modified hyperbolic functions $\widetilde{\sinh}$ and $\widetilde{\cosh}$ of \cite[Eq.~(49)]{Xie2014}, which are obtained by freezing $\sinh$ and $\cosh$ at the values $\pm\sinh\tau$ and $\cosh\tau$ for $|\xi|\ge\tau$; those functions are continuous but not differentiable at $\pm\tau$, and $\widetilde{\cosh}$ is not the derivative of $\widetilde{\sinh}$. The functions in \eqref{eq:sinhtau} coincide with $\sinh$ and $\cosh$ for $|\xi|\le\tau$, $\cosh_\tau=\sinh_\tau'$, and $\sinh_\tau$ is odd, increasing and of quadratic growth.
\end{remark}

\subsection{The truncated problem and its well-posedness}

For $\tau>0$ consider the truncated version of \eqref{eq:Phiweak}: find $\Phi_\tau\in H^1_0(\Om)$ such that
\begin{equation}\label{eq:Phitau}
a(\Phi_\tau,v)+\big(\rt(U+\Phi_\tau),v\big)_{\Ds}=0\qquad\forall v\in H^1_0(\Om).
\end{equation}
By Assumption~\ref{ass:U} and \eqref{eq:rhogrowth}, for $\Phi\in H^1_0(\Om)$ we have $|\rt(U+\Phi)|\le C_0(1+\MU+|\Phi|)^2\in L^3(\Ds)$ because $\Phi\in L^6(\Om)$, so that the nonlinear term in \eqref{eq:Phitau} is well defined for every $v\in H^1_0(\Om)\subset L^{3/2}(\Om)$. Let
\begin{equation}\label{eq:Jtau}
\Jt(v):=\tfrac12 a(v,v)+\int_{\Ds}\Pt(U+v)\dd\mathbf r,\qquad v\in H^1_0(\Om),
\end{equation}
which is finite by \eqref{eq:rhogrowth} since $H^1_0(\Om)\subset L^3(\Om)$.

\begin{theorem}\label{thm:wellposed}
Let Assumption~\ref{ass:U} and \eqref{eq:neutral} hold and let $\tau>0$. Then problem \eqref{eq:Phitau} has a unique solution $\Phi_\tau\in H^1_0(\Om)$, which is the unique minimizer of $\Jt$ over $H^1_0(\Om)$, and
\begin{equation}\label{eq:stab}
\enorm{\Phi_\tau}\le \ep^{-1/2}C_P\,\norm{\rt(U)}_{0,2,\Ds}\le \ep^{-1/2}C_PC_0|\Ds|^{1/2}(1+\MU)^2.
\end{equation}
\end{theorem}

\begin{proof}
The functional $\Jt$ is strictly convex on $H^1_0(\Om)$, since $v\mapsto a(v,v)$ is strictly convex and $\Pt$ is convex. It is coercive, since $\Pt\ge0$ gives $\Jt(v)\ge\frac12\enorm{v}^2$. It is continuous: if $v_k\to v$ in $H^1_0(\Om)$ then $v_k\to v$ in $L^3(\Om)$, and since
\[
0\le\Pt(U+v)\le C_0\big(1+\MU+|v|\big)^3\qquad\text{a.e.\ in }\Ds,
\]
the Nemytskii operator $v\mapsto\Pt(U+v)$ is continuous from $L^3(\Ds)$ to $L^1(\Ds)$ \cite{AmbrosettiProdi}. A continuous convex functional is weakly lower semicontinuous, so $\Jt$ has a unique minimizer $\Phi_\tau\in H^1_0(\Om)$ \cite[Proposition~II.1.2]{EkelandTemam}. Moreover, $\Jt$ is G\^ateaux differentiable with
\[
\Jt'(v)w = a(v,w)+\big(\rt(U+v),w\big)_{\Ds},\qquad v,w\in H^1_0(\Om),
\]
by dominated convergence, using \eqref{eq:rhogrowth} and $v,w\in L^6(\Om)$. Hence the minimizer satisfies \eqref{eq:Phitau}. Conversely, every solution of \eqref{eq:Phitau} is a critical point of the strictly convex functional $\Jt$ and therefore its unique minimizer; this proves uniqueness.

For \eqref{eq:stab}, take $v=\Phi_\tau$ in \eqref{eq:Phitau} and subtract $(\rt(U),\Phi_\tau)_{\Ds}$ from both sides:
\[
\enorm{\Phi_\tau}^2+\big(\rt(U+\Phi_\tau)-\rt(U),\Phi_\tau\big)_{\Ds} = -\big(\rt(U),\Phi_\tau\big)_{\Ds}.
\]
The second term on the left is nonnegative because $\rt$ is increasing, so by the Cauchy--Schwarz inequality, \eqref{eq:poincare} and \eqref{eq:anorm},
\[
\enorm{\Phi_\tau}^2\le\norm{\rt(U)}_{0,2,\Ds}\norm{\Phi_\tau}_{0,2}\le\norm{\rt(U)}_{0,2,\Ds}\,C_P\,\ep^{-1/2}\enorm{\Phi_\tau},
\]
which gives the first inequality in \eqref{eq:stab}; the second follows from \eqref{eq:rhogrowth}, since $|\rt(U)|\le C_0(1+\MU)^2$ a.e.\ in $\Ds$.
\end{proof}

\subsection{A priori \texorpdfstring{$L^\infty$}{L-infinity} bounds and equivalence with the original problem}

The following bound is the analogue of \cite[Theorem~2.3]{HolstSzypowskiZhu2013} and \cite[Lemma~3.1]{ChenHolstXu2007} for the problem \eqref{eq:Phitau}; its proof uses only the monotonicity of $\rt$ and the fact that $\rt(0)=0$. No smoothness of $\Gam$ beyond the Lipschitz property is required.

\begin{theorem}\label{thm:Linfty}
Let Assumption~\ref{ass:U} and \eqref{eq:neutral} hold, let $\tau>0$ and let $\Phi_\tau$ be the solution of \eqref{eq:Phitau}. Then
\begin{equation}\label{eq:Linfty}
-\MU\le\Phi_\tau\le\MU\qquad\text{a.e.\ in }\Om.
\end{equation}
The same bound holds for every solution $\Phit\in\mathcal H$ of the original problem \eqref{eq:Phiweak}.
\end{theorem}

\begin{proof}
Let $\Phi$ denote either $\Phi_\tau$ or a solution $\Phit\in\mathcal H$ of \eqref{eq:Phiweak}, and let $r$ denote $\rt$ or $\rho$ accordingly; in both cases $r$ is increasing with $r(0)=0$, and $r(U+\Phi)\in L^2(\Ds)$ (for $\rt$ this was shown above; for $\rho$ it follows from $U\in L^\infty(\Ds)$ and the definition of $\mathcal H$). Consider
\[
\phi:=(\Phi-\MU)^+=\max\{\Phi-\MU,0\}.
\]
Since $\Phi\in H^1_0(\Om)$ and $\MU\ge0$, we have $\phi\in H^1_0(\Om)$, $\phi\ge0$, and, by \cite[Theorem~A.1]{KinderlehrerStampacchia} and \cite[Lemma~7.6]{GilbargTrudinger},
\[
\nabla\phi=\nabla\Phi\ \text{ a.e.\ on }\{\Phi>\MU\},\qquad \nabla\phi=0\ \text{ a.e.\ on }\{\Phi\le\MU\}
\]
Hence $a(\Phi,\phi)=a(\phi,\phi)$. On the set $\{\phi>0\}\cap\Ds$ we have $U+\Phi\ge -\MU+\MU=0$, so $r(U+\Phi)\ge r(0)=0$ there, and therefore $(r(U+\Phi),\phi)_{\Ds}\ge0$. Taking $v=\phi$ in the weak formulation gives
\[
a(\phi,\phi) = a(\Phi,\phi) = -\big(r(U+\Phi),\phi\big)_{\Ds}\le0,
\]
so that $\enorm{\phi}=0$ and $\phi=0$, i.e., $\Phi\le\MU$ a.e. For the lower bound, let
\[
\psi:=(\Phi+\MU)^-=\min\{\Phi+\MU,0\}\in H^1_0(\Om),\qquad \psi\le0.
\]
As before $a(\Phi,\psi)=a(\psi,\psi)$, and on $\{\psi<0\}\cap\Ds$ we have $U+\Phi\le\MU-\MU=0$, so $r(U+\Phi)\le0$ and $(r(U+\Phi),\psi)_{\Ds}\ge0$. Taking $v=\psi$ gives $a(\psi,\psi)\le0$, hence $\psi=0$ and $\Phi\ge-\MU$ a.e.
\end{proof}

We can now state the main result of this section: for a truncation level above an explicit threshold, the truncated problem and the original problem have the same solution.

\begin{theorem}\label{thm:equiv}
Let Assumption~\ref{ass:U} and \eqref{eq:neutral} hold, and let
\begin{equation}\label{eq:tau0}
\tau\ge\tau_0:=2\Zmax\MU.
\end{equation}
Then the solution $\Phi_\tau$ of the truncated problem \eqref{eq:Phitau} belongs to $\mathcal H$ and is the unique solution of the original problem \eqref{eq:Phiweak}. In particular, the solution $\Phit$ of \eqref{eq:Phiweak} does not depend on $\tau\ge\tau_0$, satisfies $\norm{\Phit}_{0,\infty}\le\MU$, and
\begin{equation}\label{eq:UPhi}
\norm{U+\Phit}_{0,\infty,\Ds}\le 2\MU.
\end{equation}
\end{theorem}

\begin{proof}
By Theorem~\ref{thm:Linfty}, $|\Phi_\tau|\le\MU$ a.e.\ in $\Om$, so that
\[
|U+\Phi_\tau|\le2\MU\quad\text{and}\quad \Zmax|U+\Phi_\tau|\le\tau_0\le\tau\qquad\text{a.e.\ in }\Ds .
\]
By Lemma~\ref{lem:rhotau}(iii), $\rt(U+\Phi_\tau)=\rho(U+\Phi_\tau)$ a.e.\ in $\Ds$, so $\Phi_\tau$ satisfies \eqref{eq:Phiweak}; and $\Phi_\tau\in\mathcal H$ because it is bounded. Conversely, let $\Phit\in\mathcal H$ be any solution of \eqref{eq:Phiweak}. By the second part of Theorem~\ref{thm:Linfty}, $|\Phit|\le\MU$, hence $\rho(U+\Phit)=\rt(U+\Phit)$ and $\Phit$ solves \eqref{eq:Phitau}. By the uniqueness in Theorem~\ref{thm:wellposed}, $\Phit=\Phi_\tau$.
\end{proof}

\begin{remark}\label{rem:tau}
Theorem~\ref{thm:equiv} gives a second proof of the existence and uniqueness result \cite[Theorem~4.2]{Xie2014} for $\Phit$. It also explains why the truncation used in practice is harmless. In \cite{Xie2014} the exponentials are truncated at $\tau=85$ to avoid overflow in the first Newton iterations, with the argument that the truncation is inactive near the minimizer; by Theorem~\ref{thm:equiv}, for every $\tau\ge2\Zmax\norm{U}_{0,\infty,\Ds}$ the truncated problem coincides with the original one. In the discrete problems below we always use the truncated nonlinearity. A code that uses $\rt$ as its overflow safeguard computes exactly this discrete solution, and with the truncation the constants in the error estimates depend on $\tau$ (through $e^{\tau}$) but not on pointwise bounds of discrete solutions. The smallest constants are obtained for $\tau$ close to $\tau_0$, which only requires a bound on the potential of the solute in water in the solvent region; this bound is available once $\Psi$ has been computed. Since $|Z_i(U+\Phit)|\le\tau_0$ by \eqref{eq:UPhi}, the factor $e^{\tau_0}$ reflects the actual size of the ionic densities $e^{-Z_i u}$ and is not caused by the truncation.
\end{remark}

Henceforth we fix $\tau\ge\tau_0$ and denote by $\Phit=\Phi_\tau$ the unique solution of \eqref{eq:Phiweak} and \eqref{eq:Phitau}. For later use we record the consequences of Theorem~\ref{thm:equiv} and Lemma~\ref{lem:rhotau} that
\begin{equation}\label{eq:cbound}
0<c(\mathbf r):=\rt'\big(U(\mathbf r)+\Phit(\mathbf r)\big)\le C_1(1+2\MU)=:c_{\max}\qquad\text{for a.e.\ }\mathbf r\in\Ds .
\end{equation}

\section{Finite element discretization and error estimates}
\label{sec:fem}

\subsection{Meshes, spaces and assumptions}

From now on, we assume  that $\Om$ and $\Dp$ are polyhedral, so that the interface $\Gam$ is a closed polyhedral surface, and we let $\{\mathcal T_h\}_{h>0}$ be a family of conforming, shape-regular and quasi-uniform tetrahedral meshes of $\Om$ that are fitted to the interface, i.e., every $T\in\mathcal T_h$ satisfies $T\subset\overline{\Dp}$ or $T\subset\overline{\Ds}$. Quasi-uniformity is used only in Section~\ref{sec:twogrid}, to relate mesh sizes to numbers of unknowns. Here $h=\max_{T\in\mathcal T_h}\operatorname{diam}(T)$, and without loss of generality $h\le1$. All constants $C$ below are independent of the mesh, unless stated otherwise. (If $\Gam$ is a curved surface, the results below remain valid for meshes that resolve $\Gam$ up to $O(h^2)$, cf.\ \cite{ChenZou1998,LiMelenkWohlmuthZou2010,HolstSzypowskiZhu2013}; we do not pursue this generalization.) Let
\[
\Vh:=\{v\in C(\overline\Om): v|_T\in\mathbb P_1(T)\ \forall T\in\mathcal T_h\}\subset H^1(\Om),\qquad \Vhz:=\Vh\cap H^1_0(\Om),
\]
be the spaces of continuous piecewise linear functions. The restriction to piecewise linear elements is for ease of exposition; see Remark~\ref{rem:higherorder}.

\begin{assumption}[regularity]\label{ass:reg}
There exist $s,\sigma\in(1,2]$ such that:
\begin{enumerate}
\item[\rm (i)] the solutions $\Psi$ of \eqref{eq:Psiweak} and $\Phit$ of \eqref{eq:Phiweak} satisfy $\Psi,\Phit\in\Hs$;
\item[\rm (ii)] for every measurable $c$ with $0\le c\le c_{\max}$ a.e.\ in $\Ds$, where $c_{\max}$ is the constant in \eqref{eq:cbound}, and every $f\in L^2(\Om)$, the solution $w\in H^1_0(\Om)$ of the linear dual problem
\begin{equation}\label{eq:dual}
a(v,w)+(c\,v,w)_{\Ds}=(f,v)\qquad\forall v\in H^1_0(\Om),
\end{equation}
satisfies $w\in\Hsig$ and
\begin{equation}\label{eq:dualreg}
\norm{w}_{\Hsig}\le C_R\norm{f}_{0,2},
\end{equation}
with $C_R$ independent of $c$ and $f$ (it may depend on $c_{\max}$). (The two cases used below are $c=0$, for the linear component $\Psi$, and $c=\rt'(U+\Phit)$, for the nonlinear component $\Phit$.)
\end{enumerate}
We set
\begin{equation}\label{eq:t}
t:=\min\{s,\sigma\}-1\in(0,1].
\end{equation}
\end{assumption}

Both parts of Assumption~\ref{ass:reg} are regularity statements for linear elliptic interface problems with piecewise constant diffusion coefficient: $\Phit$ solves \eqref{eq:Phi} with the bounded right-hand side $-\rho(U+\Phit)\chi_{\Ds}\in L^\infty(\Om)$ (see \eqref{eq:UPhi}), $\Psi$ solves \eqref{eq:Psi} with data determined by $G$ and $\hat g$, and \eqref{eq:dual} is an interface problem with a bounded nonnegative reaction coefficient, which can be moved to the right-hand side since $\norm{cw}_{0,2,\Ds}\le c_{\max}C_P\ep^{-1/2}\enorm{w}\le C\norm{f}_{0,2}$ (test \eqref{eq:dual} with $v=w$ and use $c\ge0$). Assumptions of this type are common in the analysis of interface problems, see \cite[Assumption~4.3]{LiMelenkWohlmuthZou2010} and \cite[Eq.~(3.11)]{HolstSzypowskiZhu2013}.

\begin{remark}\label{rem:reg}
The exponents $s$ and $\sigma$ are determined by the regularity of the linear interface problems \eqref{eq:Psi} and \eqref{eq:dual}, hence by the geometry of $\Gam$ and $\partial\Om$ and by the ratio $\es/\ep$. For the polyhedral interfaces of this section one can in general only expect $s,\sigma\in(1,2)$, with values depending on the edge and vertex angles of $\Gam$ and on $\es/\ep$ \cite{Kellogg1975,NicaiseSandig1994}. If $\Gam$ is of class $C^2$, $\Om$ is convex and $\hat g$ is smooth, then $s=\sigma=2$ \cite{NicaiseSandig1994,ChenZou1998,Xie2014}, because $G\in C^\infty(\overline{\Ds})$ and the right-hand side of \eqref{eq:Phi} is bounded. This case lies outside the polyhedral setting assumed here, but it is the relevant one when $\Gam$ is smooth and the mesh approximates it with sufficient accuracy \cite{LiMelenkWohlmuthZou2010}. Since $\Psi$ and $\Phit$ solve interface problems of the same type with smooth or bounded data, their regularity is limited by the same edge and vertex singularities, and one can usually take $s=\sigma$. The case $\sigma<s$ occurs when the solution is smoother than in the worst case, for example for the manufactured solution of Section~\ref{sec:numerics}.
\end{remark}

\begin{lemma}\label{ass:approx}
	For every $v\in H^1_0(\Om)\cap H^{r}(\Dp\cup\Ds)$ with $1<r\le2$ there exists $v_h\in\Vhz$ such that
	\begin{equation}\label{eq:approx}
		\norm{v-v_h}_{0,2}+h\,\enorm{v-v_h}\le C_Ah^{r}\norm{v}_{H^r(\Dp\cup\Ds)}.
	\end{equation}
\end{lemma}

\begin{proof}
	The estimate \eqref{eq:approx} is a standard consequence of the elementwise
	interpolation estimates on fitted meshes, i.e., meshes whose elements satisfy
	$T\subset\overline{\Dp}$ or $T\subset\overline{\Ds}$.
	
	For $r>3/2$, the nodal interpolant $I_hv$ is well defined, since $H^r(\Dp)$ and $H^r(\Ds)$ embed into the continuous functions and $v$ is continuous across $\Gam$, and the usual
	elementwise interpolation estimates give \eqref{eq:approx} because every element lies in $\overline{\Dp}$ or in $\overline{\Ds}$. For general
	$r\in(1,2]$, the nodal interpolant need not be defined, and one instead uses a
	Scott--Zhang type quasi-interpolant \cite{ScottZhang1990} whose averaging
	patches do not cross $\Gam$ and which preserves homogeneous boundary values; the local estimates then
	yield \eqref{eq:approx} for all $r\in(1,2]$. We refer to
	\cite{ChenZou1998} and \cite[Theorem~3.5]{LiMelenkWohlmuthZou2010} for the
	details.
\end{proof}

\subsection{The linear component}

The finite element approximation of $\Psi$ is defined as in \cite[Algorithm~1]{Xie2014}: find $\Psi_h\in\Vh$ such that $\Psi_h=I_h^{\partial}(\hat g-G)$ on $\partial\Om$ and
\begin{equation}\label{eq:Psih}
a(\Psi_h,v)=(\ep-\es)\int_{\Ds}\nabla G\cdot\nabla v\dd\mathbf r\qquad\forall v\in\Vhz,
\end{equation}
where $I_h^\partial$ denotes nodal interpolation on the boundary mesh. Since $G$ is smooth near $\partial\Om$, the boundary datum $\hat g-G$ is as smooth as $\hat g$. From now on we assume that $\hat g$ is smooth, as is the case in practice ($\hat g=0$ or a smooth far-field approximation), so that nodal interpolation on $\partial\Om$ is defined. (If only $\hat g\in H^{1/2}(\partial\Om)$, $I_h^\partial$ can be replaced by the trace of a Scott--Zhang quasi-interpolant \cite{ScottZhang1990}.) Problem \eqref{eq:Psih} is a standard discretization of a linear interface problem with piecewise constant coefficients \cite{Babuska1970,ChenZou1998,PlumWieners2003,LiMelenkWohlmuthZou2010}. We include a short proof of the error estimate because of the nonhomogeneous Dirichlet datum.

\begin{lemma}\label{lem:Psih}
Let Assumption~\ref{ass:reg} hold and let $\hat g$ be smooth. Then for every mesh $\mathcal{T}_h$ of the family, the solutions of \eqref{eq:Psiweak} and \eqref{eq:Psih} satisfy
\begin{equation}\label{eq:Psierr}
\enorm{\Psi-\Psi_h}\le C_\Psi h^{s-1},\qquad
\norm{\Psi-\Psi_h}_{0,2}\le C_\Psi h^{s+\sigma-2}\le C_\Psi h^{s-1+t},
\end{equation}
with $s,\sigma$ and $t$ as in Assumption~\ref{ass:reg} and $C_\Psi$ independent of $h$.
\end{lemma}

\begin{proof}
Since $\hat g-G$ is smooth on $\partial\Om$ and $\Gam\cap\partial\Om=\emptyset$, there is a lifting $\bar\Psi\in C^\infty(\overline\Om)$ with $\bar\Psi=\hat g-G$ on $\partial\Om$ and $\bar\Psi=0$ in a neighborhood of $\Gam$. Then $I_h\bar\Psi\in\Vh$ has the boundary values $I_h^\partial(\hat g-G)$, so that $\Psi_h-I_h\bar\Psi\in\Vhz$, and
\begin{equation}\label{eq:liftint}
\enorm{\bar\Psi-I_h\bar\Psi}+h^{-1}\norm{\bar\Psi-I_h\bar\Psi}_{0,2}+h^{-1}\norm{\bar\Psi-I_h\bar\Psi}_{0,2,\partial\Om}\le C_{\bar\Psi}\,h
\end{equation}
by the standard interpolation estimates in $\Om$ and on the faces of $\partial\Om$, with $C_{\bar\Psi}$ depending on the $C^2$ norm of $\bar\Psi$. Moreover $\Psi-\bar\Psi\in H^1_0(\Om)\cap\Hs$ by Assumption~\ref{ass:reg}(i). Subtracting \eqref{eq:Psih} from \eqref{eq:Psiweak} gives the Galerkin orthogonality
\begin{equation}\label{eq:galerkin}
a(\Psi-\Psi_h,v)=0\qquad\forall v\in\Vhz .
\end{equation}

\emph{Energy estimate.} Let $v_h\in\Vhz$ be the approximation of $\Psi-\bar\Psi$ provided by Lemma~\ref{ass:approx} with $r=s$, and set $\chi_h:=I_h\bar\Psi+v_h\in\Vh$; then $\chi_h-\Psi_h\in\Vhz$, so by \eqref{eq:galerkin}
\[
\enorm{\Psi-\Psi_h}^2=a(\Psi-\Psi_h,\Psi-\chi_h)\le\enorm{\Psi-\Psi_h}\,\enorm{\Psi-\chi_h},
\]
and by \eqref{eq:liftint} and \eqref{eq:approx},
\[
\enorm{\Psi-\chi_h}\le\enorm{\bar\Psi-I_h\bar\Psi}+\enorm{(\Psi-\bar\Psi)-v_h}\le C_{\bar\Psi}h+C_Ah^{s-1}\norm{\Psi-\bar\Psi}_{\Hs}\le Ch^{s-1}.
\]

\emph{$L^2$ estimate.} Write $e:=\Psi-\Psi_h=(\bar\Psi-I_h\bar\Psi)+e^0$ with $e^0:=(\Psi-\bar\Psi)-(\Psi_h-I_h\bar\Psi)\in H^1_0(\Om)$. By \eqref{eq:liftint}, $\norm{\bar\Psi-I_h\bar\Psi}_{0,2}\le Ch^2$, so it suffices to bound $\norm{e^0}_{0,2}$. Let $w\in H^1_0(\Om)$ solve the dual problem \eqref{eq:dual} with $c=0$ and $f=e^0$, and let $w_h\in\Vhz$ be its Galerkin approximation, so that $\enorm{w-w_h}\le Ch^{\sigma-1}\norm{e^0}_{0,2}$ by Lemma~\ref{ass:approx} and Assumption~\ref{ass:reg}(ii). Using \eqref{eq:galerkin} in the form $a(e^0,w_h)=a(e,w_h)-a(\bar\Psi-I_h\bar\Psi,w_h)=-a(\bar\Psi-I_h\bar\Psi,w_h)$, we obtain
\begin{equation}\label{eq:PsiL2split}
\norm{e^0}_{0,2}^2=a(e^0,w)=a(e^0,w-w_h)-a(\bar\Psi-I_h\bar\Psi,w_h-w)-a(\bar\Psi-I_h\bar\Psi,w).
\end{equation}
The first two terms are bounded by $C\big(\enorm{e^0}+\enorm{\bar\Psi-I_h\bar\Psi}\big)\enorm{w-w_h}\le Ch^{s-1}h^{\sigma-1}\norm{e^0}_{0,2}$, where $\enorm{e^0}\le\enorm{e}+\enorm{\bar\Psi-I_h\bar\Psi}\le Ch^{s-1}$ by the energy estimate. For the last term we integrate by parts in $\Dp$ and $\Ds$: since $-\nabla\cdot(\eps\nabla w)=e^0$ in $\Dp\cup\Ds$ and the flux $\eps\,\partial w/\partial\mathbf n$ is continuous across $\Gam$, while $\bar\Psi-I_h\bar\Psi$ vanishes near $\Gam$,
\[
a(\bar\Psi-I_h\bar\Psi,w)=\big(\bar\Psi-I_h\bar\Psi,e^0\big)+\es\int_{\partial\Om}(\bar\Psi-I_h\bar\Psi)\,\frac{\partial w}{\partial\mathbf n}\dd s .
\]
The volume term is bounded by $Ch^2\norm{e^0}_{0,2}$ by \eqref{eq:liftint}. For the boundary term, if $\sigma>3/2$ the trace $\partial w/\partial\mathbf n\in L^2(\partial\Om)$ satisfies $\norm{\partial w/\partial\mathbf n}_{0,2,\partial\Om}\le C\norm{w}_{H^\sigma(\Ds)}\le C\norm{e^0}_{0,2}$, and \eqref{eq:liftint} bounds the term by $Ch^{2}\norm{e^0}_{0,2}$. If $1<\sigma\le3/2$, the conormal derivative is defined by Green's formula as an element of $H^{\sigma-3/2}(\partial\Om)$ with norm bounded by $C\norm{e^0}_{0,2}$, and interpolating the $L^2(\partial\Om)$ and $H^{1/2}(\partial\Om)$ interpolation estimates for $\bar\Psi-I_h\bar\Psi$ on the faces of $\partial\Om$ gives $\norm{\bar\Psi-I_h\bar\Psi}_{H^{3/2-\sigma}(\partial\Om)}\le Ch^{\sigma+1/2}$, so that the term is bounded by $Ch^{\sigma+1/2}\norm{e^0}_{0,2}$. Since $s+\sigma-2\le\min\{\sigma+1/2,\ 2\}$ for $s,\sigma\in(1,2]$, all contributions in \eqref{eq:PsiL2split} are bounded by $Ch^{s+\sigma-2}\norm{e^0}_{0,2}$, which proves the $L^2$ estimate; the last inequality in \eqref{eq:Psierr} uses $h\le1$ and $\sigma-1\ge t$.
\end{proof}

Since $h\le1$ and $s,\sigma>1$, Lemma~\ref{lem:Psih} and \eqref{eq:anorm} give, for every mesh of the family,
\begin{equation}\label{eq:Lambda}
\norm{\Psi_h}_{H^1(\Om)}\le\norm{\Psi}_{H^1(\Om)}+\norm{\Psi-\Psi_h}_{0,2}+\ep^{-1/2}\enorm{\Psi-\Psi_h}
\le\norm{\Psi}_{H^1(\Om)}+\big(1+\ep^{-1/2}\big)C_\Psi=:\Lambda .
\end{equation}
The constant $\Lambda$ does not depend on the mesh; in Section~\ref{sec:twogrid} it is used for both the fine mesh $\mathcal T_h$ and the coarse mesh $\mathcal T_H$. We set $U_h:=G+\Psi_h$.

\subsection{The nonlinear component: discrete problem and stability}

In the discrete problem for $\Phit$, the function $U$ is not available and is replaced by an approximation. To treat the fine and coarse meshes of the two-grid method simultaneously, we consider a generic mesh $\mathcal T_h$ and a generic datum
\begin{equation}\label{eq:Ustar}
U^\ast := G+\Psi^\ast,\qquad \Psi^\ast\in H^1(\Om),\quad \norm{\Psi^\ast}_{H^1(\Om)}\le\Lambda,
\end{equation}
where $\Lambda>0$ is a fixed constant. With $\Lambda$ as in \eqref{eq:Lambda}, both $\Psi^\ast=\Psi_h$ and $\Psi^\ast=\Psi_H$ are admissible. The discrete nonlinear problem reads: find $\Phih\in\Vhz$ such that
\begin{equation}\label{eq:Phih}
a(\Phih,v)+\big(\rt(U^\ast+\Phih),v\big)_{\Ds}=0\qquad\forall v\in\Vhz;
\end{equation}
its well-posedness is established in Lemma~\ref{lem:discrete} below. We write
\begin{equation}\label{eq:delta}
\delta:=\norm{\Psi-\Psi^\ast}_{0,2,\Ds}=\norm{U-U^\ast}_{0,2,\Ds}
\end{equation}
for the size of the data perturbation. The nonlinearity in \eqref{eq:Phih} is the truncated one, in accordance with Remark~\ref{rem:tau}, and that $U^\ast+\Phih$ need not be bounded pointwise.

\begin{lemma}\label{lem:discrete}
Let Assumption~\ref{ass:U} and \eqref{eq:neutral} hold, $\tau>0$, and let $U^\ast$ be as in \eqref{eq:Ustar}. Then \eqref{eq:Phih} has a unique solution $\Phih\in\Vhz$, and
\begin{equation}\label{eq:dstab}
\enorm{\Phih}\le\ep^{-1/2}C_P\norm{\rt(U^\ast)}_{0,2,\Ds}\le C_\Lambda,
\end{equation}
where $C_\Lambda$ depends only on $\tau,\beta,Z_i,\bar c_i,\ep,\Om$, $\norm{G}_{H^1(\Ds)}$ and $\Lambda$, but not on $h$.
\end{lemma}

\begin{proof}
The functional
\[
\Jt^\ast(v):=\tfrac12a(v,v)+\int_{\Ds}\Pt(U^\ast+v)\dd\mathbf r,\qquad v\in\Vhz,
\]
is well defined on $\Vhz$ because $U^\ast\in L^6(\Ds)$ (recall that $G\in C^\infty(\overline{\Ds})$ and $\Psi^\ast\in H^1(\Om)\hookrightarrow L^6(\Om)$) and $v\in L^6(\Om)$, and the argument of Theorem~\ref{thm:wellposed} carries over to the finite-dimensional space $\Vhz$: $\Jt^\ast$ is strictly convex, continuous and coercive, so it has a unique minimizer, which is the unique solution of \eqref{eq:Phih}. The first inequality in \eqref{eq:dstab} is proved exactly as \eqref{eq:stab}. For the second, \eqref{eq:rhogrowth} gives
\[
\norm{\rt(U^\ast)}_{0,2,\Ds}\le C_0\norm{(1+|U^\ast|)^2}_{0,2,\Ds}=C_0\norm{1+|U^\ast|}_{0,4,\Ds}^2\le C_0\big(|\Ds|^{1/4}+\norm{U^\ast}_{0,4,\Ds}\big)^2,
\]
and
\[
\norm{U^\ast}_{0,4,\Ds}\le C\norm{U^\ast}_{H^1(\Ds)}\le C\big(\norm{G}_{H^1(\Ds)}+\Lambda\big)
\]
by the Sobolev embedding on the Lipschitz domain $\Ds$.
\end{proof}

We shall repeatedly use the following elementary consequence of the mean value theorem and Lemma~\ref{lem:rhotau}(ii): for all $\xi_1,\xi_2\in\R$,
\begin{equation}\label{eq:mvt}
|\rt(\xi_1)-\rt(\xi_2)|\le C_1\big(1+|\xi_1|+|\xi_1-\xi_2|\big)|\xi_1-\xi_2|,
\end{equation}
which follows from $\rt(\xi_1)-\rt(\xi_2)=\rt'(\zeta)(\xi_1-\xi_2)$ with $\zeta$ between $\xi_1$ and $\xi_2$, so that $|\zeta|\le|\xi_1|+|\xi_1-\xi_2|$.

\subsection{Energy-norm error estimate}

\begin{theorem}\label{thm:energy}
Let Assumptions~\ref{ass:U} and \ref{ass:reg}(i) hold, let \eqref{eq:neutral} and \eqref{eq:tau0} hold, and let $U^\ast$ be as in \eqref{eq:Ustar}. There exist $h_0>0$ and $C>0$, depending only on $\tau,\beta,Z_i,\bar c_i,\ep,\es,\Om,\MU,\Lambda$, $\norm{G}_{H^1(\Ds)}$, $\norm{\Psi}_{H^1(\Om)}$ and $\norm{\Phit}_{\Hs}$, such that for all $h\le h_0$ the solution $\Phih$ of \eqref{eq:Phih} satisfies
\begin{equation}\label{eq:energy}
\enorm{\Phit-\Phih}\le C\big(h^{s-1}+\delta\big),
\end{equation}
where $\delta$ is defined in \eqref{eq:delta}.
\end{theorem}

\begin{proof}
Set $e:=\Phit-\Phih$. Subtracting \eqref{eq:Phih} from \eqref{eq:Phitau} (recall $\Phit=\Phi_\tau$ by Theorem~\ref{thm:equiv}) gives the error equation
\begin{equation}\label{eq:erreq}
a(e,\chi)+(T_1+T_2,\chi)_{\Ds}=0\qquad\forall\chi\in\Vhz,
\end{equation}
where
\[
T_1:=\rt(U+\Phit)-\rt(U+\Phih),\qquad T_2:=\rt(U+\Phih)-\rt(U^\ast+\Phih).
\]
Let $v_h\in\Vhz$ be arbitrary and $\eta:=\Phit-v_h$. Then $v_h-\Phih=e-\eta\in\Vhz$, and by \eqref{eq:erreq},
\begin{equation}\label{eq:split}
\begin{aligned}
\enorm{e}^2&=a(e,\eta)+a(e,e-\eta)\\
&=a(e,\eta)-(T_1,e)_{\Ds}-(T_2,e)_{\Ds}+(T_1,\eta)_{\Ds}+(T_2,\eta)_{\Ds}.
\end{aligned}
\end{equation}
Since $\rt$ is increasing,
\[
T_1e=\big(\rt(U+\Phit)-\rt(U+\Phit-e)\big)e\ge0\qquad\text{a.e.\ in }\Ds,
\]
so $-(T_1,e)_{\Ds}\le0$. We bound the remaining terms.

\emph{The term $T_1$.} Applying \eqref{eq:mvt} with $\xi_1=U+\Phit$ and $\xi_2=U+\Phih$, so that $\xi_1-\xi_2=e$, and using \eqref{eq:UPhi},
\[
|T_1|\le C_1(1+2\MU+|e|)|e|\qquad\text{a.e.\ in }\Ds .
\]
Hence, by H\"older's inequality and \eqref{eq:poincare}--\eqref{eq:anorm},
\begin{equation}\label{eq:T1}
\begin{aligned}
|(T_1,\eta)_{\Ds}|&\le C_1(1+2\MU)\norm{e}_{0,2,\Ds}\norm{\eta}_{0,2,\Ds}+C_1\norm{e}_{0,4,\Ds}^2\norm{\eta}_{0,2,\Ds}\\
&\le C\big(\enorm{e}+\enorm{e}^2\big)\norm{\eta}_{0,2}.
\end{aligned}
\end{equation}

\emph{The term $T_2$.} Applying \eqref{eq:mvt} with $\xi_1=U+\Phih$ and $\xi_2=U^\ast+\Phih$, so that $\xi_1-\xi_2=U-U^\ast=\Psi-\Psi^\ast$,
\[
|T_2|\le C_1\big(1+\MU+|\Phih|+|\Psi-\Psi^\ast|\big)|\Psi-\Psi^\ast|\qquad\text{a.e.\ in }\Ds .
\]
For any $w\in H^1_0(\Om)$, H\"older's inequality with exponents $(3,2,6)$ gives
\begin{equation}\label{eq:T2}
\begin{aligned}
|(T_2,w)_{\Ds}|&\le C_1\norm{1+\MU+|\Phih|+|\Psi-\Psi^\ast|}_{0,3,\Ds}\,\norm{\Psi-\Psi^\ast}_{0,2,\Ds}\,\norm{w}_{0,6}\\
&\le C\,\delta\,\enorm{w},
\end{aligned}
\end{equation}
where we used
\[
\norm{\Phih}_{0,3}\le C\enorm{\Phih}\le CC_\Lambda,\qquad
\norm{\Psi-\Psi^\ast}_{0,3,\Ds}\le C\big(\norm{\Psi}_{H^1(\Om)}+\Lambda\big),\qquad
\norm{w}_{0,6}\le C\enorm{w},
\]
by Lemma~\ref{lem:discrete}, the Sobolev embedding $H^1(\Om)\hookrightarrow L^3(\Om)$, and \eqref{eq:poincare}--\eqref{eq:anorm}, respectively.

\emph{Conclusion.} Inserting \eqref{eq:T1} and \eqref{eq:T2} (with $w=e$ and $w=\eta$) into \eqref{eq:split} and using $a(e,\eta)\le\enorm e\enorm\eta$ yields
\begin{equation}\label{eq:pre}
\enorm{e}^2\le\enorm{e}\big(\enorm{\eta}+C\norm{\eta}_{0,2}+C\delta\big)+C\enorm{e}^2\norm{\eta}_{0,2}+C\delta\enorm{\eta}.
\end{equation}
Now choose $v_h$ according to Lemma~\ref{ass:approx} with $r=s$ and $v=\Phit\in H^1_0(\Om)\cap\Hs$, so that
\[
\enorm{\eta}\le C_Ah^{s-1}\norm{\Phit}_{\Hs},\qquad \norm{\eta}_{0,2}\le C_Ah^{s}\norm{\Phit}_{\Hs},
\]
and let $h_0>0$ be such that
\[
C\,C_A\,h_0^{s}\norm{\Phit}_{\Hs}\le\tfrac12,
\]
where $C$ is the constant in front of $\enorm{e}^2\norm{\eta}_{0,2}$ in \eqref{eq:pre}. Then for $h\le h_0$ the term $C\enorm{e}^2\norm{\eta}_{0,2}$ can be absorbed into the left-hand side of \eqref{eq:pre}, and we obtain, using $h^s\le h^{s-1}$ and Young's inequality $\delta h^{s-1}\le\frac12(\delta^2+h^{2(s-1)})$,
\[
\tfrac12\enorm{e}^2\le C\big(h^{s-1}+\delta\big)\enorm{e}+C\big(h^{2(s-1)}+\delta^2\big).
\]
If $x\ge0$ satisfies $x^2\le bx+c$ with $b,c\ge0$, then $x\le b+\sqrt c$. Applying this with $x=\enorm e$ gives \eqref{eq:energy}.
\end{proof}

\begin{remark}
The truncation and the bound \eqref{eq:UPhi} enter the proof as follows. The pointwise bound on $U+\Phit$ gives the linear part of $T_1$ a bounded coefficient. The quadratic growth of $\rt$ makes the rest of $T_1$ quadratic in $e$, so it is controlled by $\norm{e}_{0,4}^2\lesssim\enorm{e}^2$; this term is absorbed for small $h$ because it is multiplied by $\norm{\eta}_{0,2}=O(h^s)$. Unlike \cite[Theorem~3.5]{HolstSzypowskiZhu2013}, no pointwise bound on $\Phih$ is used, and hence no condition on the mesh beyond shape regularity is needed. In return we need $h\le h_0$, as is usual for this kind of argument (cf.\ \cite{BankHolstSzypowskiZhu2011}).
\end{remark}

\subsection{\texorpdfstring{$L^2$ and $L^4$}{L2 and L4} error estimates}

\begin{theorem}[$L^2$ estimate]\label{thm:L2}
Under the hypotheses of Theorem~\ref{thm:energy}, let in addition Assumption~\ref{ass:reg}(ii) hold. Then for all $h\le h_0$,
\begin{equation}\label{eq:L2}
\norm{\Phit-\Phih}_{0,2}\le C\big(h^{\sigma-1}\enorm{\Phit-\Phih}+\enorm{\Phit-\Phih}^2+\delta\big),
\end{equation}
with $C$ depending on the same quantities as in Theorem~\ref{thm:energy} and on $C_R$.
\end{theorem}

\begin{proof}
Let $e=\Phit-\Phih$ and let $w\in H^1_0(\Om)$ be the solution of the dual problem \eqref{eq:dual} with $f=e$, where $c=\rt'(U+\Phit)$ satisfies $0<c\le c_{\max}$ by \eqref{eq:cbound}. The bilinear form $B(v,w):=a(v,w)+(cv,w)_{\Ds}$ satisfies
\[
B(v,v)\ge\enorm{v}^2,\qquad B(v,w)\le\big(1+c_{\max}C_P^2\ep^{-1}\big)\enorm v\,\enorm w\qquad\forall v,w\in H^1_0(\Om).
\]
Let $w_h\in\Vhz$ be the Galerkin approximation of $w$, $B(v_h,w_h)=(e,v_h)$ for all $v_h\in\Vhz$. By C\'ea's lemma, Lemma~\ref{ass:approx} with $r=\sigma$ and Assumption~\ref{ass:reg}(ii) with $c=\rt'(U+\Phit)$,
\begin{equation}\label{eq:dualerr}
\enorm{w-w_h}\le C\inf_{v_h\in\Vhz}\enorm{w-v_h}\le Ch^{\sigma-1}\norm{w}_{\Hsig}\le Ch^{\sigma-1}\norm{e}_{0,2},
\end{equation}
and, taking $v_h=w_h$ in the discrete dual problem,
\begin{equation}\label{eq:whbound}
\enorm{w_h}^2\le B(w_h,w_h)=(e,w_h)\le\norm{e}_{0,2}\norm{w_h}_{0,2}\le C\norm{e}_{0,2}\enorm{w_h},
\quad\text{so that}\quad \enorm{w_h}\le C\norm{e}_{0,2}.
\end{equation}
Taking $v=e$ in \eqref{eq:dual} (recall that $f=e$) and using the error equation \eqref{eq:erreq} with $\chi=w_h$, we get
\begin{align}
\norm{e}_{0,2}^2 &= B(e,w)=B(e,w-w_h)+a(e,w_h)+(ce,w_h)_{\Ds}\notag\\
&= B(e,w-w_h)+\big(ce-T_1,w_h\big)_{\Ds}-\big(T_2,w_h\big)_{\Ds}.\label{eq:L2split}
\end{align}
For the first term, \eqref{eq:dualerr} gives
\[
|B(e,w-w_h)|\le C\enorm e\,\enorm{w-w_h}\le Ch^{\sigma-1}\enorm e\,\norm{e}_{0,2}.
\]
For the second term we use Taylor's formula with integral remainder and $\Phih=\Phit-e$: a.e.\ in $\Ds$,
\[
\rt(U+\Phit-e)=\rt(U+\Phit)-\rt'(U+\Phit)e+e^2\int_0^1(1-\vartheta)\,\rt''(U+\Phit-\vartheta e)\dd\vartheta,
\]
so that
\[
ce-T_1=\rt'(U+\Phit)e-\big(\rt(U+\Phit)-\rt(U+\Phit-e)\big)=e^2\int_0^1(1-\vartheta)\,\rt''(U+\Phit-\vartheta e)\dd\vartheta .
\]
Since $|\rt''|\le C_2$ by \eqref{eq:rhogrowth}, we get $|ce-T_1|\le\frac12C_2e^2$, and H\"older's inequality, \eqref{eq:poincare} and \eqref{eq:whbound} yield
\[
|(ce-T_1,w_h)_{\Ds}|\le\tfrac12C_2\norm{e}_{0,4}^2\norm{w_h}_{0,2}\le C\enorm{e}^2\enorm{w_h}\le C\enorm{e}^2\norm{e}_{0,2}.
\]
For the third term, \eqref{eq:T2} and \eqref{eq:whbound} give
\[
|(T_2,w_h)_{\Ds}|\le C\delta\enorm{w_h}\le C\delta\norm{e}_{0,2}.
\]
Collecting the three bounds in \eqref{eq:L2split} and dividing by $\norm{e}_{0,2}$ proves \eqref{eq:L2}.
\end{proof}

The quadratic term in \eqref{eq:L2} is the place where the bounded second derivative of the truncated nonlinearity is used; with the exponential itself this step would need pointwise bounds on $\Phih$, as in \cite[Theorem~3.7]{HolstSzypowskiZhu2013}.

\begin{corollary}[$L^2$ and $L^4$ estimates]\label{cor:L4}
Under the hypotheses of Theorem~\ref{thm:L2}, suppose in addition that $\delta\le C_\delta h^{s-1+t}$. Then for all $h\le h_0$, with $C$ depending additionally on $C_\delta$,
\begin{equation}\label{eq:L2L4}
\enorm{\Phit-\Phih}\le Ch^{s-1},\qquad \norm{\Phit-\Phih}_{0,2}\le Ch^{s-1+t},\qquad \norm{\Phit-\Phih}_{0,4}\le Ch^{s-1+t/4},
\end{equation}
where $t$ is given by \eqref{eq:t}.
\end{corollary}

\begin{proof}
The first estimate follows from \eqref{eq:energy} since $\delta\le C_\delta h^{s-1+t}\le C_\delta h^{s-1}$. Inserting it into \eqref{eq:L2} gives
\[
\norm{\Phit-\Phih}_{0,2}\le C\big(h^{\sigma-1+s-1}+h^{2(s-1)}+h^{s-1+t}\big)\le Ch^{s-1+t},
\]
since $t\le\sigma-1$ and $t\le s-1$. Finally, \eqref{eq:GN} and \eqref{eq:anorm} give
\[
\norm{\Phit-\Phih}_{0,4}\le C_L\norm{\Phit-\Phih}_{0,2}^{1/4}\norm{\nabla(\Phit-\Phih)}_{0,2}^{3/4}\le C\big(h^{s-1+t}\big)^{1/4}\big(h^{s-1}\big)^{3/4}=Ch^{s-1+t/4}.\qedhere
\]
\end{proof}

Applying Corollary~\ref{cor:L4} with $\Psi^\ast=\Psi_h$ and Lemma~\ref{lem:Psih}, so that $\delta=\norm{\Psi-\Psi_h}_{0,2,\Ds}\le C_\Psi h^{s-1+t}$, we obtain the error estimates for the finite element solution decomposition scheme of \cite[Algorithm~1]{Xie2014}.

\begin{corollary}[error of the fine-mesh solution]\label{cor:fine}
Let Assumptions~\ref{ass:U} and \ref{ass:reg} hold, let $\hat g$ be smooth, let \eqref{eq:neutral} and \eqref{eq:tau0} hold, and let $\Phih$ be the solution of \eqref{eq:Phih} with $U^\ast=U_h=G+\Psi_h$. Then for $h\le h_0$ the finite element solution $u_h:=G+\Psi_h+\Phih$ satisfies
\begin{equation}\label{eq:uh}
\enorm{u-u_h}\le Ch^{s-1},\qquad\norm{u-u_h}_{0,2}\le Ch^{s-1+t}.
\end{equation}
\end{corollary}

\begin{proof}
Since $u-u_h=(\Psi-\Psi_h)+(\Phit-\Phih)$, the claim follows from \eqref{eq:Psierr} and \eqref{eq:L2L4}; note that $\Psi-\Psi_h\notin H^1_0(\Om)$ in general, but $\enorm{\cdot}$ is still defined on $H^1(\Om)$ by \eqref{eq:a}.
\end{proof}

\begin{remark}\label{rem:higherorder}
If $\Vh$ consists of continuous piecewise polynomials of degree $k\ge1$ and the solutions have piecewise $H^{s}$ regularity with $1<s\le k+1$, then Lemma~\ref{ass:approx} holds with $r=s$ \cite{LiMelenkWohlmuthZou2010}, and all the results of this section and of Section~\ref{sec:twogrid} hold with the same proofs. For molecular surfaces, however, the regularity of $\Phit$ is limited, and higher-order elements pay off only if the interface is also approximated to higher order \cite{Xie2014}.
\end{remark}

\section{The two-grid algorithm and its analysis}
\label{sec:twogrid}

\subsection{The algorithm}

Let $\mathcal T_H$ be a coarse mesh with mesh size $H$ from the family introduced in Section~\ref{sec:fem}, and let $\mathcal T_h$, $h<H$, be obtained from $\mathcal T_H$ by (possibly repeated) uniform refinement that keeps it in the family (e.g., red refinement of tetrahedra), so that $\mathcal T_h$ is fitted to $\Gam$ and the finite element spaces are nested, $\VHz\subset\Vhz$. The two-grid solution decomposition scheme is described in Algorithm~\ref{alg:tg}. In Step~3 the datum $\Psi^\sharp$ used to assemble the coarse nonlinear problem may be either the fine-mesh approximation $\Psi_h$ or a coarse-mesh approximation $\Psi_H\in\VH$ defined by \eqref{eq:Psih} on $\mathcal T_H$; both choices are covered by the analysis below (see Remark~\ref{rem:variants} for a discussion).

\begin{algorithm}[htb]
\caption{Two-grid PBE solution decomposition scheme}
\label{alg:tg}
\begin{algorithmic}[1]
\Require atomic positions $\mathbf r_j$ and charges $z_j$; ionic species $(Z_i,\bar c_i)$; permittivities $\ep,\es$; boundary value $\hat g$; truncation level $\tau$; meshes $\mathcal T_H$, $\mathcal T_h$.
\State Compute $G$ and $\nabla G$ from \eqref{eq:G}.
\State \textbf{(Fine linear solve.)} Find $\Psi_h\in\Vh$ with $\Psi_h=I_h^\partial(\hat g-G)$ on $\partial\Om$ satisfying \eqref{eq:Psih}; set $U_h:=G+\Psi_h$. Choose $\Psi^\sharp\in\{\Psi_h,\Psi_H\}$, where $\Psi_H\in\VH$ is computed by \eqref{eq:Psih} on $\mathcal T_H$ if this variant is chosen, and set $U^\sharp:=G+\Psi^\sharp$.
\State \textbf{(Coarse nonlinear solve.)} Find $\PhiH\in\VHz$ such that
\begin{equation}\label{eq:coarse}
a(\PhiH,v)+\big(\rt(U^\sharp+\PhiH),v\big)_{\Ds}=0\qquad\forall v\in\VHz .
\end{equation}
\State \textbf{(Fine linear solve: one Newton correction.)} Find $\Phitg\in\Vhz$ such that for all $v\in\Vhz$
\begin{equation}\label{eq:fine}
a(\Phitg,v)+\big(\rt'(U_h+\PhiH)\,\Phitg,v\big)_{\Ds}=\big(\rt'(U_h+\PhiH)\,\PhiH-\rt(U_h+\PhiH),v\big)_{\Ds}.
\end{equation}
\Ensure $\utg:=G+\Psi_h+\Phitg$.
\end{algorithmic}
\end{algorithm}

The coarse problem \eqref{eq:coarse} is problem \eqref{eq:Phih} on $\mathcal T_H$ with datum $\Psi^\ast=\Psi^\sharp$. It has a unique solution by Lemma~\ref{lem:discrete} and can be solved, for example, by the modified Newton minimization scheme of \cite[Algorithm~2]{Xie2014}; for the truncated nonlinearity this is a descent method for the strictly convex functional $\Jt^\ast$ of Lemma~\ref{lem:discrete} with $U^\ast=U^\sharp$ and, with a suitable line search, converges from any initial guess. Writing $\Phitg=\PhiH+d_h$, the fine problem \eqref{eq:fine} is equivalent to the Newton correction equation
\begin{equation}\label{eq:newton}
a(d_h,v)+\big(\rt'(U_h+\PhiH)\,d_h,v\big)_{\Ds}=-\Big[a(\PhiH,v)+\big(\rt(U_h+\PhiH),v\big)_{\Ds}\Big]\qquad\forall v\in\Vhz,
\end{equation}
whose right-hand side is the residual of the coarse solution in the fine problem \eqref{eq:Phih} with $U^\ast=U_h$. The bilinear form $B_h(v,w):=a(v,w)+(\rt'(U_h+\PhiH)v,w)_{\Ds}$ on the left of \eqref{eq:newton} is symmetric and, since $\rt'>0$, it satisfies
\begin{equation}\label{eq:Bh}
\enorm{v}^2\le B_h(v,v)\le\Big(1+C_1\ep^{-1}C_P^2\big(1+\norm{U_h+\PhiH}_{0,\infty,\Ds}\big)\Big)\enorm{v}^2\qquad\forall v\in\Vhz,
\end{equation}
where the upper bound uses \eqref{eq:rhogrowth} and the fact that $U_h+\PhiH$ is bounded on $\Ds$ ($G$ is smooth on $\overline{\Ds}$ and $\Psi_h+\PhiH$ is piecewise linear). Thus only two symmetric positive definite systems are solved on the fine mesh, one for $\Psi_h$ and one for $d_h$, and both can be solved by preconditioned conjugate gradients with multilevel preconditioners that are robust with respect to the jump of $\eps$ \cite{XuZhu2008}. The coefficient $\rt'(U_h+\PhiH)$ in \eqref{eq:fine} is bounded on each element but may be large where the truncation is active; the analysis does not use an upper bound for it. Since $\PhiH\in\Vhz$, Step~4 is one Newton step on the fine mesh started from the coarse solution, cf.\ \cite{Xu1996,HolstSzypowskiZhu2013}.

\subsection{Error analysis}

Throughout this subsection, $\Phih\in\Vhz$ denotes the solution of the fine-mesh nonlinear problem \eqref{eq:Phih} with $U^\ast=U_h$, i.e., the ``exact'' fine-mesh finite element solution, and $u_h=G+\Psi_h+\Phih$ as in Corollary~\ref{cor:fine}. The first lemma is the analogue of \cite[Lemma~4.1]{HolstSzypowskiZhu2013} and shows that the two-grid solution differs from the fine-mesh solution by a quantity that is quadratic in the coarse-mesh error.

\begin{lemma}\label{lem:tg}
Let Assumption~\ref{ass:U} and \eqref{eq:neutral} hold and $\tau>0$. Let $\PhiH\in\VHz\subset\Vhz$ be any function and let $\Phitg\in\Vhz$ be defined by \eqref{eq:fine}. Then
\begin{equation}\label{eq:tglemma}
\enorm{\Phih-\Phitg}\le\tfrac12C_2C_P\ep^{-1/2}\,\norm{\Phih-\PhiH}_{0,4}^2 .
\end{equation}
\end{lemma}

\begin{proof}
Write $b(\xi):=\rt(U_h+\xi)$ for the fine-mesh nonlinearity, so that $b'(\xi)=\rt'(U_h+\xi)$ and $|b''(\xi)|\le C_2$ a.e.\ in $\Ds$ by \eqref{eq:rhogrowth}. Let $\chi\in\Vhz$ and put $\theta:=\Phih-\PhiH\in\Vhz$. By Taylor's formula with integral remainder, a.e.\ in $\Ds$,
\[
b(\Phih)=b(\PhiH)+b'(\PhiH)\theta+\int_0^1(1-\vartheta)\,b''(\PhiH+\vartheta\theta)\,\theta^2\dd\vartheta .
\]
Testing the fine problem \eqref{eq:Phih} (with $U^\ast=U_h$) with $\chi$ and inserting this expansion,
\begin{equation}\label{eq:taylor}
\begin{aligned}
0&=a(\Phih,\chi)+\big(b(\Phih),\chi\big)_{\Ds}\\
&=a(\PhiH,\chi)+\big(b(\PhiH),\chi\big)_{\Ds}+a(\theta,\chi)+\big(b'(\PhiH)\theta,\chi\big)_{\Ds}+R(\chi),
\end{aligned}
\end{equation}
where
\[
R(\chi):=\int_0^1(1-\vartheta)\big(b''(\PhiH+\vartheta\theta)\theta^2,\chi\big)_{\Ds}\dd\vartheta
\]
satisfies
\begin{equation}\label{eq:R}
|R(\chi)|\le\tfrac12C_2\norm{\theta}_{0,4}^2\norm{\chi}_{0,2}.
\end{equation}
On the other hand, \eqref{eq:fine} can be written as
\[
a(\Phitg-\PhiH,\chi)+\big(b'(\PhiH)(\Phitg-\PhiH),\chi\big)_{\Ds}=-a(\PhiH,\chi)-\big(b(\PhiH),\chi\big)_{\Ds}.
\]
Adding this identity to \eqref{eq:taylor} and rearranging gives, with $\omega:=\Phih-\Phitg=\theta-(\Phitg-\PhiH)$,
\[
a(\omega,\chi)+\big(b'(\PhiH)\omega,\chi\big)_{\Ds}=-R(\chi)\qquad\forall\chi\in\Vhz.
\]
Choosing $\chi=\omega$, using $b'>0$, \eqref{eq:R}, \eqref{eq:poincare} and \eqref{eq:anorm},
\[
\enorm{\omega}^2\le a(\omega,\omega)+\big(b'(\PhiH)\omega,\omega\big)_{\Ds}=-R(\omega)\le\tfrac12C_2\norm{\theta}_{0,4}^2\,C_P\ep^{-1/2}\enorm{\omega},
\]
which gives \eqref{eq:tglemma} after dividing by $\enorm\omega$.
\end{proof}

We now combine Lemma~\ref{lem:tg} with the $L^4$ estimates of Corollary~\ref{cor:L4}, applied on both meshes.

\begin{theorem}\label{thm:twogrid}
Let Assumptions~\ref{ass:U} and \ref{ass:reg} hold, let $\hat g$ be smooth, let \eqref{eq:neutral} and \eqref{eq:tau0} hold, and assume that the coarse datum in Algorithm~\ref{alg:tg} satisfies
\begin{equation}\label{eq:sharp}
\norm{\Psi^\sharp}_{H^1(\Om)}\le\Lambda,\qquad \norm{\Psi-\Psi^\sharp}_{0,2,\Ds}\le C_\Psi H^{s-1+t},
\end{equation}
which holds for $\Psi^\sharp=\Psi_h$ and for $\Psi^\sharp=\Psi_H$ by \eqref{eq:Lambda} and Lemma~\ref{lem:Psih}. Then there are constants $h_0>0$ and $C>0$, independent of $h$ and $H$, such that for all $h<H\le h_0$ the two-grid solution $\utg$ of Algorithm~\ref{alg:tg} satisfies
\begin{equation}\label{eq:tgmain}
\enorm{u-\utg}\le C\Big(h^{s-1}+H^{2(s-1)+t/2}\Big),\qquad t=\min\{s,\sigma\}-1 .
\end{equation}
\end{theorem}

\begin{proof}
By the triangle inequality and Corollary~\ref{cor:fine},
\begin{equation}\label{eq:tri}
\enorm{u-\utg}\le\enorm{u-u_h}+\enorm{\Phih-\Phitg}\le Ch^{s-1}+\enorm{\Phih-\Phitg}.
\end{equation}
By Lemma~\ref{lem:tg} and the triangle inequality in $L^4(\Om)$,
\begin{equation}\label{eq:quad}
\enorm{\Phih-\Phitg}\le C\norm{\Phih-\PhiH}_{0,4}^2\le 2C\Big(\norm{\Phit-\Phih}_{0,4}^2+\norm{\Phit-\PhiH}_{0,4}^2\Big),
\end{equation}
where we used $(x+y)^2\le2(x^2+y^2)$. The function $\Phih$ is the solution of \eqref{eq:Phih} on $\mathcal T_h$ with $\Psi^\ast=\Psi_h$, for which $\delta\le C_\Psi h^{s-1+t}$ by Lemma~\ref{lem:Psih}; the function $\PhiH$ is the solution of \eqref{eq:Phih} on $\mathcal T_H$ with $\Psi^\ast=\Psi^\sharp$, for which $\delta\le C_\Psi H^{s-1+t}$ by \eqref{eq:sharp}. Since $h<H\le h_0$, Corollary~\ref{cor:L4} applies on both meshes and gives
\[
\norm{\Phit-\Phih}_{0,4}\le Ch^{s-1+t/4},\qquad \norm{\Phit-\PhiH}_{0,4}\le CH^{s-1+t/4}.
\]
The constants in Corollary~\ref{cor:L4} depend on the datum only through $\Lambda$ and $C_\Psi$, and are therefore the same for both meshes. Inserting these bounds into \eqref{eq:quad} and then into \eqref{eq:tri}, and using $h<H$, yields \eqref{eq:tgmain}.
\end{proof}

\begin{corollary}\label{cor:H}
Under the hypotheses of Theorem~\ref{thm:twogrid}, if the coarse mesh size satisfies
\begin{equation}\label{eq:Hchoice}
H\le C h^{\gamma},\qquad \gamma:=\frac{s-1}{2(s-1)+t/2},
\end{equation}
then $\enorm{u-\utg}\le Ch^{s-1}$, i.e., the two-grid solution has the same asymptotic accuracy as the fine-mesh solution $u_h$. If $\sigma\ge s$, then $t=s-1$ and $\gamma=2/5$ for every $s\in(1,2]$; in particular $H\sim h^{2/5}$ suffices when $s=\sigma$, which includes the case of full regularity $s=\sigma=2$.
\end{corollary}

\begin{remark}
If $\sigma<s$, then $t=\sigma-1<s-1$ and $2/5<\gamma<1/2$, with $\gamma\to1/2$ as $\sigma\to1$; the coarse mesh then has to be somewhat finer relative to the fine mesh. The exponent $\gamma$ is the same as in \cite[Remark~4.5]{HolstSzypowskiZhu2013}, where it was obtained under a discrete maximum principle; here no angle condition on $\mathcal T_h$ and $\mathcal T_H$ is needed, and the nonlinearity contains the discrete datum $U_h$. For quasi-uniform meshes and $\gamma=2/5$, the numbers of unknowns satisfy $N_H\sim H^{-3}\sim h^{-6/5}\sim N_h^{2/5}$; for example, $N_h\approx10^6$ corresponds to $N_H$ of a few hundred. In practice the coarse mesh still has to resolve the molecular surface, and the constants in \eqref{eq:tgmain} depend on the geometry, so a somewhat finer coarse mesh may be preferable.
\end{remark}

\begin{remark}[the two variants of Step~3]\label{rem:variants}
Theorem~\ref{thm:twogrid} covers both $\Psi^\sharp=\Psi_h$ and $\Psi^\sharp=\Psi_H$, because the proof compares $\PhiH$ directly with $\Phit$ and only needs the perturbation of the datum in the coarse problem to be of the order $H^{s-1+t}$ of the coarse $L^2$ error. The fine Newton correction \eqref{eq:fine} always uses $U_h$. The choice $\Psi^\sharp=\Psi_H$ costs one extra linear solve on the coarse mesh, but the nonlinear coarse problem is then assembled entirely on $\mathcal T_H$. With $\Psi^\sharp=\Psi_h$ no coarse linear solve is needed, but the integrals $(\rt(U_h+\PhiH),v)_{\Ds}$ have to be evaluated by quadrature on the fine mesh, since $U_h$ is piecewise linear on $\mathcal T_h$; this costs $O(N_h)$ operations per coarse Newton iteration, but no fine-mesh solve. In both cases the fine-mesh work consists of two linear solves, whereas the scheme of \cite{Xie2014} needs one fine linear solve for $\Psi_h$ and one per Newton iteration.
\end{remark}

\subsection{The symmetric \texorpdfstring{$1{:}1$}{1:1} electrolyte}

For a symmetric $1{:}1$ ionic solution, $\rho(\xi)=\kappa^2\sinh\xi$ with $\kappa^2=2\beta I_s$, and $\rt=\kappa^2\sinh_\tau$, $\rt'=\kappa^2\cosh_\tau$ with $\sinh_\tau,\cosh_\tau$ defined in \eqref{eq:sinhtau}. Here $\Zmax=1$ and $\tau_0=2\MU$. The coarse and fine problems of Algorithm~\ref{alg:tg} read: find $\PhiH\in\VHz$ with
\[
a(\PhiH,v)+\kappa^2\big(\sinh_\tau(U^\sharp+\PhiH),v\big)_{\Ds}=0\qquad\forall v\in\VHz,
\]
and find $\Phitg\in\Vhz$ with
\begin{multline*}
a(\Phitg,v)+\kappa^2\big(\cosh_\tau(U_h+\PhiH)\Phitg,v\big)_{\Ds}\\
=\kappa^2\big(\cosh_\tau(U_h+\PhiH)\PhiH-\sinh_\tau(U_h+\PhiH),v\big)_{\Ds}\qquad\forall v\in\Vhz .
\end{multline*}
Apart from the use of $C^2$ instead of $C^0$ truncations ($\sinh_\tau,\cosh_\tau$ instead of $\widetilde{\sinh},\widetilde{\cosh}$ in \cite[Eq.~(49)]{Xie2014}), the coarse problem is the nonlinear problem solved by the PBE finite element package of \cite{Xie2014}, and the fine problem is one step of its modified Newton scheme \cite[Eq.~(40)]{Xie2014} with unit step length, taken on the fine mesh from the coarse solution.

\begin{corollary}\label{cor:11}
Let $n=2$, $Z_1=-Z_2=1$, $\bar c_1=\bar c_2=I_s>0$, let Assumptions~\ref{ass:U} and \ref{ass:reg} hold with $\hat g$ smooth, and let $\tau\ge2\MU$. Then the conclusions of Theorem~\ref{thm:twogrid} and Corollary~\ref{cor:H} hold for the two-grid scheme above. Moreover, the solution $\Phit$ of \eqref{eq:Phiweak} satisfies $|\Phit|\le\MU$ a.e.\ in $\Om$, so that the dimensionless potential satisfies $\norm{u}_{0,\infty,\Ds}=\norm{U+\Phit}_{0,\infty,\Ds}\le 2\MU$.
\end{corollary}

\begin{remark}[a linear initial guess is not needed]
Two initial guesses for the Newton minimization are suggested in \cite{Xie2014}: $\Phit^{(0)}=0$, or the solution of the linearized problem $-\es\Delta\Phit+\kappa^2\Phit=-\kappa^2U$ in $\Ds$. In the two-grid scheme Newton's method runs only on the coarse mesh, where iterations are cheap, so the initial guess matters much less; on the fine mesh the starting point is the coarse solution, and by Theorem~\ref{thm:twogrid} one step is enough. If the nonlinear fine-mesh solution itself is wanted, further Newton steps can be taken on the fine mesh starting from $\Phitg$, cf.\ the experiments in \cite{HolstSzypowskiZhu2013,YingFanLiLu2021}.
\end{remark}

\begin{remark}[on the constants]
The constants in Theorems~\ref{thm:energy}--\ref{thm:twogrid} depend exponentially on the truncation level through the factor $e^{\tau}$ in $C_0,C_1,C_2$ of \eqref{eq:rhogrowth}, and on $\MU$ through \eqref{eq:UPhi}. Since $\MU$ is an essential supremum of the potential of the solute in water over the solvent region, including points close to the molecular surface, $e^{\tau_0}$ may be large for highly charged solutes; this exponential dependence reflects the size of the ionic densities $e^{-Z_iu}$ near the surface and also appears in \cite{ChenHolstXu2007,HolstSzypowskiZhu2013}. The level $\tau=85$ used in \cite{Xie2014} to prevent overflow is allowed by the theory but gives larger constants; the convergence rates, in particular the exponents in \eqref{eq:tgmain}, do not depend on the choice of $\tau\ge\tau_0$.
\end{remark}

\section{Numerical verification}
\label{sec:numerics}

We test the convergence rates of Sections~\ref{sec:fem} and \ref{sec:twogrid}, the quadratic dependence of the two-grid correction on the coarse error (Lemma~\ref{lem:tg}), the choice $H\sim h^{2/5}$, and Theorem~\ref{thm:equiv} on a problem with a manufactured solution. The test isolates the nonlinear component $\Phit$; the linear component $\Psi$ is not computed, and $U$ is prescribed analytically. Since the nonlinear term is integrated with the vertex quadrature rule (see below), only the nodal values of $U$ enter the discrete problems, i.e., the discrete problem on each mesh automatically uses the interpolated datum $I_hU$ (resp.\ $I_HU$); the data perturbation $\delta=\norm{U-I_hU}_{0,2,\Ds}=O(h^2)$ of \eqref{eq:delta} is therefore present implicitly, and the two variants of Step~3 of Algorithm~\ref{alg:tg} coincide. All experiments were run on an Apple Mac Studio (Apple M4 Max, 16 cores, 64 GB),  the code is available from the authors.

\subsection{Setting}

Let $\Om=(-1,1)^3$, $\Dp=(-\tfrac12,\tfrac12)^3$, $\ep=2$, $\es=80$, and consider a symmetric $1{:}1$ electrolyte with $\kappa^2=40$ in $\Ds$. We solve
\begin{equation}\label{eq:numprob}
-\nabla\cdot(\eps\nabla\Phi)+\kappa^2\chi_{\Ds}\sinh(U+\Phi)=f\ \text{ in }\Om,\qquad \Phi=0\ \text{ on }\partial\Om,
\end{equation}
with the homogeneous interface conditions of \eqref{eq:Phi}, where the source $f$ is chosen so that the exact solution (denoted by $\Phi$ in this section) is
\[
\Phi(x,y,z)=\tfrac14\,q(x)q(y)q(z),\qquad q(x)=\cos(2\pi x)-1,
\]
and the datum is $U(x,y,z)=2\cos(\pi x)\cos(\pi y)\cos(\pi z)$, so that $\MU=2$. Since $q'(\pm\tfrac12)=0$, the normal derivative of $\Phi$ vanishes on $\Gam$ and both interface conditions hold; $f=-\eps\Delta\Phi+\kappa^2\chi_{\Ds}\sinh(U+\Phi)$ is piecewise smooth and discontinuous across $\Gam$. The presence of $f$ does not affect the error analysis except for an additional linear term on the right-hand sides; it does, however, invalidate the a priori bound of Theorem~\ref{thm:Linfty}, see Section~\ref{sec:num:tau}. The manufactured solution is smooth, so $s=2$. For the dual problem \eqref{eq:dual} on the cubic interface with $\es/\ep=40$ we cannot expect $\sigma=2$ (cf.\ Remark~\ref{rem:reg}), so Corollary~\ref{cor:H} guarantees full accuracy only for $H\le Ch^{\gamma}$ with some $\gamma\in(2/5,1/2)$; the observed $L^2$ rates below are nevertheless consistent with $t=1$.

The domain is partitioned into $N^3$ cubes of side $h=2/N$, each split into six Kuhn tetrahedra (we identify $h$ with the cube side, which is proportional to the element diameter); for $N$ divisible by $4$ the mesh is fitted to $\Gam$, and the meshes for $N$ and $2N$ are nested, so that $\VHz\subset\Vhz$ and the prolongation of $\PhiH$ to the fine mesh in Step~4 of Algorithm~\ref{alg:tg} is exact. The nonlinear term and the load are integrated with the vertex quadrature rule (mass lumping), which makes the Newton Jacobian equal to the stiffness matrix plus a diagonal matrix. The truncated functions $\sinh_\tau,\cosh_\tau$ of \eqref{eq:sinhtau} are used with $\tau=85$ unless stated otherwise; by Theorem~\ref{thm:equiv} this does not alter the solution (see also Section~\ref{sec:num:tau}). All linear systems on meshes with more than $2\cdot10^4$ unknowns are solved by the conjugate gradient method with a smoothed-aggregation algebraic multigrid preconditioner to a relative residual of $10^{-12}$; smaller systems are solved directly. Newton's method (undamped) on a given mesh is started from zero and stopped when the Euclidean norm of the residual is below $10^{-12}$ relative to that of the load; this applies to the fine-mesh reference solutions, to the coarse solves in Algorithm~\ref{alg:tg}, and to the runs of Section~\ref{sec:num:tau}. Errors are computed with the four-point degree-two quadrature rule and are reported relative to the exact solution: for a function $v$ we write
\[
e_E(v):=\frac{\enorm{v}}{\enorm{\Phi}},\qquad e_2(v):=\frac{\norm{v}_{0,2}}{\norm{\Phi}_{0,2}},\qquad e_4(v):=\frac{\norm{v}_{0,4}}{\norm{\Phi}_{0,4}},
\]
with $\enorm{\Phi}=68.41$, $\norm{\Phi}_{0,2}=1.299$ and $\norm{\Phi}_{0,4}=1.272$. Differences of discrete solutions, such as $\Phih-\Phitg$, are scaled by the same norms of $\Phi$.

\subsection{Fine-mesh finite element solution}

Table~\ref{tab:fine} reports the relative errors of the finite element solution $\Phih$ of \eqref{eq:Phih} with $U^\ast=I_hU$ computed by Newton's method on the fine mesh alone. The energy norm error is $O(h)$ and the $L^2$ error is $O(h^2)$, in agreement with Corollary~\ref{cor:L4} for $s=2$ and $t=1$. At $h=1/32$ the relative errors are about $8.7\%$ in the energy norm and $0.36\%$ in $L^2$; the energy error is the larger one because $\enorm{\cdot}$ contains the factor $\es=80$ in $\Ds$ and the exact solution has frequency $2\pi$. The $L^4$ error also converges as $O(h^2)$, which is faster than the rate $h^{s-1+t/4}=h^{5/4}$ guaranteed by \eqref{eq:L2L4}; the latter is obtained from the Gagliardo--Nirenberg inequality \eqref{eq:GN} and is not sharp for smooth solutions. Since the nonlinearity is weak compared with the diffusion term, Newton's method converges in three or four iterations from the zero initial guess.

\begin{table}[htb]
\centering
\caption{Relative errors of the fine-mesh finite element solution $\Phih$ (full Newton iteration).}
\label{tab:fine}
\begin{tabular}{crccccccc}
\toprule
$h$ & DOFs & $e_E(\Phi-\Phih)$ & rate & $e_2(\Phi-\Phih)$ & rate & $e_4(\Phi-\Phih)$ & rate & Newton its.\\
\midrule
$1/4$ & 343 & 6.835e$-$01 & -- & 2.259e$-$01 & -- & 1.926e$-$01 & -- & 4 \\
$1/8$ & 3375 & 3.448e$-$01 & 0.99 & 5.655e$-$02 & 2.00 & 4.791e$-$02 & 2.01 & 3 \\
$1/16$ & 29791 & 1.729e$-$01 & 1.00 & 1.420e$-$02 & 1.99 & 1.208e$-$02 & 1.99 & 3 \\
$1/32$ & 250047 & 8.655e$-$02 & 1.00 & 3.557e$-$03 & 2.00 & 3.030e$-$03 & 2.00 & 3 \\
\bottomrule
\end{tabular}
\end{table}

\subsection{Two-grid solution with \texorpdfstring{$H\sim h^{2/5}$}{H ~ h\^{}(2/5)}}

Table~\ref{tab:tgrule} and Figure~\ref{fig:results}(a) compare the two-grid solution $\Phitg$ of Algorithm~\ref{alg:tg} with the fine-mesh solution $\Phih$ for coarse meshes chosen by the rule $H\approx h^{2/5}$ (which Corollary~\ref{cor:H} covers only when $\sigma\ge s$): since the available coarse meshes have $H\in\{1/2,1/4,\dots\}$, we take for $H$ the dyadic mesh size nearest to $h^{2/5}$; the value of $h^{2/5}$ is listed for comparison. For $h\le1/16$ the energy norm errors of the two-grid and the fine-mesh solutions agree to three significant digits, and both converge as $O(h)$. The additional error $e_E(\Phih-\Phitg)$ caused by replacing the fine Newton iteration by one correction is between two and three orders of magnitude below the discretization error for $h\le1/16$. The last column gives the total time of the two-grid solve (coarse Newton iteration plus one fine linear solve, excluding mesh and load assembly in both cases); for $h=1/32$ it is $2.6$~s against $8.2$~s for the full fine-mesh Newton iteration. The gain is moderate because the full Newton iteration needs only three steps in this test; for problems that require more (damped) Newton steps, the saving should be larger.

\begin{table}[htb]
\centering
\caption{Two-grid solution with $H\approx h^{2/5}$: relative energy-norm error of the fine-mesh solution $\Phih$ and of the two-grid solution $\Phitg$, relative $L^2$ error of $\Phitg$, relative distance between the two, coarse Newton iterations, and wall time of the two-grid solve.}
\label{tab:tgrule}
\setlength{\tabcolsep}{4pt}
\begin{tabular}{ccccccccc}
\toprule
$h$ & $H$ & $h^{2/5}$ & $e_E(\Phi-\Phih)$ & $e_E(\Phi-\Phitg)$ & $e_2(\Phi-\Phitg)$ & $e_E(\Phih-\Phitg)$ & coarse its. & time (s)\\
\midrule
$1/4$ & $1/2$ & 0.574 & 6.8353e$-$01 & 7.1349e$-$01 & 2.792e$-$01 & 7.98e$-$02 & 6 & 0.01 \\
$1/8$ & $1/2$ & 0.435 & 3.4480e$-$01 & 3.5485e$-$01 & 9.734e$-$02 & 4.65e$-$02 & 6 & 0.06 \\
$1/16$ & $1/4$ & 0.330 & 1.7294e$-$01 & 1.7296e$-$01 & 1.438e$-$02 & 3.15e$-$04 & 4 & 0.28 \\
$1/32$ & $1/4$ & 0.250 & 8.6551e$-$02 & 8.6560e$-$02 & 3.732e$-$03 & 2.82e$-$04 & 4 & 2.64 \\
\bottomrule
\end{tabular}
\end{table}

\begin{figure}[htb]
\centering
\includegraphics[width=\textwidth]{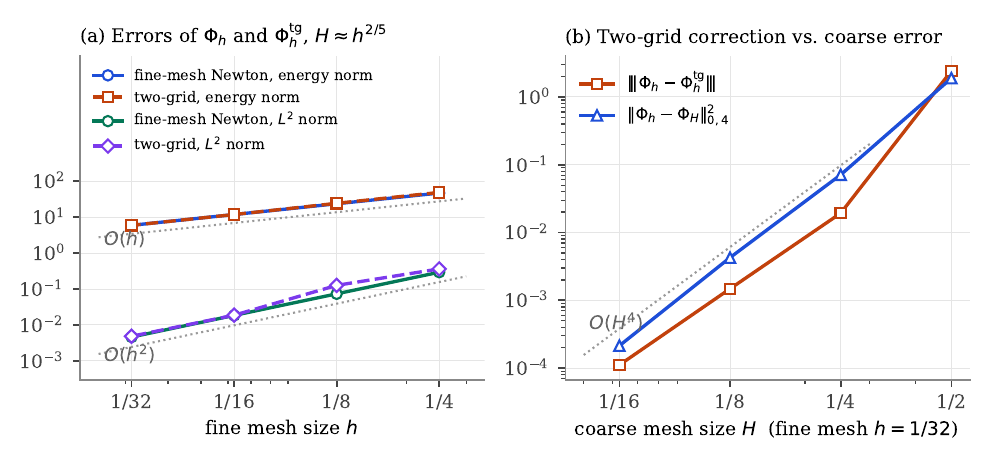}
\caption{(a) Relative energy and $L^2$ errors of the fine-mesh solution $\Phih$ and of the two-grid solution $\Phitg$ with $H\approx h^{2/5}$ (Tables~\ref{tab:fine} and \ref{tab:tgrule}). (b) For the fixed fine mesh $h=1/32$: the relative two-grid correction error $e_E(\Phih-\Phitg)$ and the squared relative $L^4$ coarse error $e_4(\Phih-\PhiH)^2$ as functions of $H$ (Table~\ref{tab:fixedh}).}
\label{fig:results}
\end{figure}

\subsection{Dependence on the coarse mesh for fixed \texorpdfstring{$h$}{h}}
\label{sec:num:fixed}

To examine Lemma~\ref{lem:tg} directly, we fix the fine mesh $h=1/32$ and vary the coarse mesh, $H=1/2,1/4,1/8,1/16$. Table~\ref{tab:fixedh} and Figure~\ref{fig:results}(b) report the relative coarse error $e_4(\Phih-\PhiH)$, the relative two-grid correction error $e_E(\Phih-\Phitg)$, the ratio of the corresponding absolute quantities, and the total error. The coarse $L^4$ error converges approximately as $O(H^2)$, and the two-grid correction error with observed order about $3.7$, so that the ratio $\enorm{\Phih-\Phitg}/\norm{\Phih-\PhiH}_{0,4}^2$ stays bounded (between $0.27$ and $0.52$ in the asymptotic range), as predicted by \eqref{eq:tglemma}. The observed order is higher than the guaranteed order $H^{2(s-1)+t/2}=H^{5/2}$ in \eqref{eq:tgmain}, for the same reason as in Table~\ref{tab:fine}: the $L^4$ estimate \eqref{eq:L2L4} is a worst-case bound. Both quantities are measured against the fine-mesh solution $\Phih$ with $h=1/32$, so that the last row ($H=2h$) is at the limit of the asymptotic regime, which explains the slight drift of the ratio. Already for $H=1/4$, i.e., with $343$ interior coarse degrees of freedom against $250\,047$ on the fine mesh, the energy error of the two-grid solution exceeds that of the fine-mesh solution by only about $0.01\%$.

\begin{table}[htb]
\centering
\caption{Fixed fine mesh $h=1/32$ and varying coarse mesh: relative coarse error in the energy and $L^4$ norms, relative two-grid correction error, the ratio $\enorm{\Phih-\Phitg}/\norm{\Phih-\PhiH}_{0,4}^2$ of absolute quantities as in \eqref{eq:tglemma}, and relative energy error of $\Phitg$ (the fine-mesh solution has $e_E(\Phi-\Phih)$=8.6551e-02).}
\label{tab:fixedh}
\begin{tabular}{cccccccc}
\toprule
$H$ & $e_E(\Phih-\PhiH)$ & $e_4(\Phih-\PhiH)$ & rate & $e_E(\Phih-\Phitg)$ & rate & ratio & $e_E(\Phi-\Phitg)$\\
\midrule
$1/2$ & 2.06e+00 & 1.080e+00 & -- & 3.533e$-$02 & -- & 1.28 & 9.4700e$-$02 \\
$1/4$ & 6.78e$-$01 & 2.097e$-$01 & 2.36 & 2.822e$-$04 & 6.97 & 0.27 & 8.6560e$-$02 \\
$1/8$ & 3.34e$-$01 & 5.134e$-$02 & 2.03 & 2.164e$-$05 & 3.70 & 0.35 & 8.6552e$-$02 \\
$1/16$ & 1.50e$-$01 & 1.150e$-$02 & 2.16 & 1.625e$-$06 & 3.73 & 0.52 & 8.6551e$-$02 \\
\bottomrule
\end{tabular}
\end{table}

\subsection{Dependence on the truncation level}
\label{sec:num:tau}

Finally, we solve the coarse problem on the mesh $H=1/4$ with Newton's method from the zero initial guess for several truncation levels $\tau$ and compare the results with the solution for $\tau=85$. Here $\max_{\Ds}|U+\PhiH|=2.508$. For every $\tau\ge3$ the truncation is never activated during the iteration and the computed solutions are bitwise identical, in agreement with Theorem~\ref{thm:equiv}. Because of the source term $f$ the bound $2\MU=4$ does not apply here, but the argument of Theorem~\ref{thm:equiv}, applied to the discrete problem, works with the threshold $\Zmax\max|U+\PhiH|=2.508$ in its place. For $\tau=2.5$, just below this value, the truncation is activated and the solutions differ by $3\cdot10^{-9}$ in the maximum norm; for $\tau=2$ and $\tau=1.5$ the differences grow to $6.4\cdot10^{-4}$ and $3.5\cdot10^{-3}$, so the computed solution changes as soon as $\tau$ drops below the threshold in Theorem~\ref{thm:equiv}. Newton's method needed four iterations for $\tau\ge2.5$ and three for $\tau\le2$.

\section{Concluding remarks}
\label{sec:conclusion}

We have analyzed a two-grid finite element method for the nonlinear Poisson--Boltzmann equation based on the solution decomposition $u=G+\Psi+\Phit$ of \cite{Xie2014}. The nonlinear problem for $\Phit$ is solved only on a coarse mesh, and the fine-mesh approximation is obtained by one Newton correction, so that only two symmetric positive definite linear systems are solved on the fine mesh. The analysis uses a $C^2$ truncation of the exponential that grows quadratically above a level $\tau$. For $\tau\ge2\Zmax\norm{G+\Psi}_{L^\infty(\Ds)}$ the truncated problem has the same solution as the original one; this gives a new proof of well-posedness, shows that the overflow safeguard used in practice does not change the solution, and leads to error estimates that need no discrete maximum principle and no angle condition on the meshes. Under piecewise $H^{s}$ regularity of the primal problem and $H^{\sigma}$ regularity of a linear dual problem we proved $\enorm{u-\utg}\le C(h^{s-1}+H^{2(s-1)+t/2})$ with $t=\min\{s,\sigma\}-1$, so that $H\sim h^{2/5}$ suffices whenever $\sigma\ge s$. The estimates include the error from replacing $U=G+\Psi$ by its finite element approximation, allow the coarse problem to be assembled with coarse-mesh data, and hold for any number of ionic species. A test with a manufactured solution in three dimensions agrees with these results.

Several extensions seem possible. Higher-order elements, and meshes that only approximate a curved interface, could be treated by combining the arguments of Section~\ref{sec:fem} with the approximation results of \cite{LiMelenkWohlmuthZou2010}. The same approach should apply to other regularized forms of the PBE, such as the two-component splitting of \cite{ChenHolstXu2007,Holst2012}. For the size-modified PBE, whose nonlinearity is bounded, no truncation would be needed. The combination of solution decomposition and two-grid linearization can also be used for the Poisson--Nernst--Planck equations of ion channels \cite{XieChao2020,ChaoXie2021}, where the coupling between potential and concentrations takes the place of the exponential nonlinearity; we plan to study this in future work, together with experiments on proteins using the PBE finite element package of \cite{Xie2014}.

\section*{Author Contributions}

All authors contributed to the analysis and to the writing of the manuscript. Z.~Chao also designed the numerical test and performed the computations.

\section*{Acknowledgments}

Z. Chao was supported by an AMS-Simons Research Enhancement Grant for PUI Faculty.

\section*{Conflicts of Interest}

The authors declare no conflicts of interest.

\section*{Data Availability Statement}

The code of Section~\ref{sec:numerics} is available from the corresponding author upon reasonable request.

\end{document}